\documentclass{amsart} 
\usepackage{amssymb, latexsym} 
\usepackage{amsmath,amsthm,amsfonts,mathabx} 
\usepackage[bookmarks,colorlinks=true]{hyperref}
\usepackage[T1]{fontenc}
\usepackage{currvita}

\theoremstyle{plain} 
\newtheorem{theorem}{Theorem}[section] 
\newtheorem{corollary}[theorem]{Corollary} 
\newtheorem{proposition}[theorem]{Proposition}

\newtheorem{lemma}[theorem]{Lemma}

\theoremstyle{definition} 
\newtheorem{definition}[theorem]{Definition}

\theoremstyle{remark} 
\newtheorem{notation}[theorem]{Notation}
\newtheorem{remark}[theorem]{Remark}

\newtheorem{porism}[theorem]{Porism}
\newtheorem{fact}[theorem]{Fact}

\numberwithin{equation}{section}

\newcommand{\Cyl}[1]{\,\,\,\left\ldbrack{#1}\right\rdbrack}
\newcommand{\seq}[1]{{\left\langle{#1}\right\rangle}}

\newcommand{\rest}[1]{\! \upharpoonright_{#1}} 

\newcommand{\conc}{\hat{\,\,}}

\newcommand{\andd}{\,\,\,\&\,\,\,}

\DeclareMathOperator{\dom}{dom}
\DeclareMathOperator{\range}{range}
\DeclareMathOperator{\id}{id}
\DeclareMathOperator{\cf}{cf}

\newcommand{\w}{\omega}
\newcommand{\s}{\sigma}
\renewcommand{\le}{\leqslant}
\renewcommand{\ge}{\geqslant}

\renewcommand{\preceq}{\preccurlyeq}
\newcommand{\vphi}{\varphi}

\newcommand\force{\Vdash}

\newcommand \compembed{\lessdot}
\newcommand \elementary{\prec}
\newcommand \Powerset{{\mathcal P}}

\newcommand{\meet}{\wedge}
\newcommand{\join}{\vee}
\newcommand \uu{\mathbf u}
\newcommand \PP{\mathbb P}
\newcommand \QQ{\mathbb Q}
\newcommand \RR{\mathbb R}
\newcommand \CC{\mathbb C}
\newcommand \BB{\mathbb B}
\newcommand \SSS{\mathbb S}
\newcommand \R{\mathbb R}

\DeclareMathOperator{\Ult}{Ult}

\newcommand{\PPP}{\mathfrak P}

\newcommand{\Cohen}{\CC}
\newcommand{\starr}{\!*}

\newcommand{\Shelah}{\texttt{\textup{Sh}}}
\DeclareMathOperator{\supp}{supp}

\newcommand{\funnyDiamond}{\hat\diamond}
\newcommand{\iplus}{\ltimes}

\title{Models of Cohen measurability}

\author{Noam Greenberg} 
\address{Department of Mathematics, Victoria University of Wellington, Wellington, New Zealand}
\email{greenberg@msor.vuw.ac.nz}
\urladdr{\url{http://homepages.mcs.vuw.ac.nz/~greenberg/}}

\author{Saharon Shelah}
\email{shelah@math.huji.ac.il}
\address{The Hebrew University of Jerusalem \\ Einstein Institute of
Mathematics \\ Edmond J. Safra Campus, Givat Ram \\ Jerusalem 91904,
Israel}
\address{Department of Mathematics \\ Hill Center-Busch Campus \\ Rutgers,
The State University of New Jersey \\ 110 Frelinghuysen Road
\\ Piscataway, NJ 08854-8019 USA}

\thanks{Greenberg was supported by a Rutherford Discovery Fellowship from the Royal Society of New Zealand, and by the Turing Centenary research project "Mind,
Mechanism and Mathematics", funded by the John Templeton Foundation. Shelah would like to thank the Israel Science Foundation for
partial support of this research (Grant no. 1053/11). 
Publication 1039 on his list.}

\begin{document}

\maketitle

\begin{abstract}
	We show that in contrast with the Cohen version of Solovay's model, it is consistent for the continuum to be Cohen-measurable and for every function to be continuous on a non-meagre set. 
\end{abstract}

\section{Introduction}

In~\cite{GitikShelah}, Gitik and Shelah answered a question of Fremlin's (\cite[P1]{Fremlin}). They showed that it is possible to construct a model of set theory in which the continuum is real-valued measurable in a way that is different from Solovay's original construction of such a model  (\cite{Solovay}). Solovay used a measurable-length sequence of random reals. Fremlin's general question is, what properties of Solovay's model are artefacts of the construction, and which follow from the fact that the continuum is real-valued measurable. The paper~\cite{FGS} extends this line of investigation. It gives yet another construction of a model with the continuum being real-valued measurable, and isolates a measure-theoretic property which differentiates between this model and Solovay's. 

It is natural to ask what happens when measure is replaced by category. The analogue of Solovay's forcing is the addition of a measurable-length sequence (or more) of Cohen reals, equivalently forcing with the open subsets of the product space $2^\kappa$ for $\kappa$ measurable. The analogue of real-valued measurability, which holds in this model is the following. 

\begin{definition}\label{def:RVC}
	A cardinal $\kappa$ is \emph{Cohen measurable} if there is a $\kappa$-complete ideal $I$ on $\kappa$ such that $\Powerset(\kappa)/I$ is isomorphic to a Cohen algebra.
\end{definition}
Here a \emph{Cohen algebra} is the completion of a notion of forcing adding a certain amount of Cohen reals (the finite support product of copies of $2^{<\w}$), equivalently, the Boolean algebra of regular open subsets of a space $2^X$ for some $X$.

\smallskip

One would expect that a modification of the notion of forcing from~\cite{FGS} would yield a model in which the continuum is Cohen measurable. This follows the intuition that category is easier than measure. It turns out however that this is not easily done; that construction heavily relies on the existence of measures on measure algebras. In other words, with category we have fewer tools because we cannot say ``how much more or less meagre'' is one open set compared to another; there is no real number value that can answer such a question. 

In this paper we give a new construction of a model in which the continuum is Cohen measurable. Rather than drawing on~\cite{FGS}, we adopt a technique from~\cite{Sh473}. The simplicity of category compared to measure is reflected in the kind of difference between the new model and the model obtained by adding a sequence of Cohen reals. Whereas the statement in~\cite{FGS} was somewhat ad-hoc, we now obtain the natural property of the continuum, which was first shown to be consistent in~\cite{Sh473}.

\begin{theorem}\label{thm:main}
	Let $\kappa$ be a measurable cardinal such that $2^\kappa = \kappa^+$. Then in a forcing extension, $\kappa = 2^{\aleph_0}$ is Cohen measurable, and every function $f\colon 2^\w \to 2^\w$ is continuous on a non-meagre set. 
\end{theorem}

In contrast, in~\cite{Sh473}, Shelah showed that in a model obtained by adding Cohen reals, some function from $2^\w$ to $2^\w$ will not be continuous on any non-meagre set.

\subsection{Proof of Theorem~\ref{thm:main}}
%
%
%
%
%
%
The general idea is to use a finite-support forcing iteration $\bar \PP= \seq{\PP_\alpha,\QQ_\alpha}$ of length $\kappa$ (where $\kappa$ is a measurable cardinal) which is ``mostly Cohen''. We will specify a stationary subset $S$ of $S^\kappa_{\aleph_1} = \{ \alpha<\kappa\,:\, \cf(\alpha)  = \aleph_1\}$. This will be the set of locations $\alpha$ at which we can choose $\QQ_\alpha$ to be a notion of forcing other than Cohen forcing. The intention is to use a carefully chosen variant of Shelah's notion of forcing from~\cite{Sh473}, to add instances of continuity on a non-meagre set. Exactly how we choose them will be determined using a diamond-like sequence on $S$. (We will not be able to ensure the full diamond on~$S$, so we will define a weaker combinatorial notion $\funnyDiamond(S)$ (Definition~\ref{def:funnyDiamond}) which we suse instead.) The guessing power of the $\funnyDiamond(S)$-sequence will be sufficient to guess, for each function $F\colon 2^\w\to 2^\w$ in $V^{\PP_\kappa}$, sufficiently much about $F$, so that at some point $\alpha \in S$, $\QQ_\alpha$ will add both the definition $\Psi$ of a continuous function, and a non-meagre set $A$ on which $F$ will equal $\Psi$. The fact that the iteration $\bar \PP$ is mostly Cohen will be also used to show that the non-meagreness of $A$ is preserved from step $\alpha+1$ all the way up to step $\kappa$. Further, $S$ will be made sufficiently sparse, so that the deviations on $S$ from Cohen forcing do not aggregate too badly to prevent us from making $\kappa$ Cohen measurable. One aspect of this is that $\QQ_\alpha$ will be determined by only few of the~$\QQ_\beta$ ($\beta<\alpha$). We will call these ``iterations with restricted memory''. 

The proper definition of what we do at steps $\alpha\in S$ actually relies on a structural analysis of what we have done up to that stage. For that reason we define for each ordinal $\delta$, classes $\PPP_\delta(S)$ of forcing iterations of length $\delta$ which \emph{could} be the one chosen up to stage $\delta$. Once we develop the general theory of these iterations, we can then use the $\funnyDiamond(S)$ sequence and give non-circular instructions at each step, how to choose the next~$\QQ_\delta$. In the construction of $\bar \PP$ we will only use the fixed stationary set $S\subset \kappa$ and naturally use only ordinals $\delta\le \kappa$. However in the verification that in $V^{\PP_\kappa}$, $\kappa$ is Cohen measurable, we also need to consider the extension of $\bar \PP$ by an elementary embedding $j$ witnessing the measurability of $\kappa$. In particular we will need to consider $\PPP_{j(\kappa)}(j(S))$, where in our ambient universe $j(\kappa)$ will not even be inaccessible and $j(S)$ will not be a stationary subset of $j(\kappa)$. Hence we will give a general definition of the classes $\PPP_\delta(S)$, for any ordinal $\delta$ and any subset $S\subseteq S^\delta_{\aleph_1}$ (Definition~\ref{def:S-iteration}). The restriction $S\subset \delta$ does not conflict with the plan for the recursive definition of the eventual $\PP$; for $S\subset \kappa$ and $\delta<\kappa$, $\PPP_\delta(S)$ will only depend on $S\cap \delta$. 

\medskip

We will show the following. For~\eqref{item:Cohen} below, we say that a set $S$ \emph{reflects no-where} at a set of ordinals $A$ if for all limit $\gamma\in A$ of uncountable cofinality, $S\cap \gamma$ is not stationary in $\gamma$. We note that if $S\subseteq S^{\delta}_{\aleph_1}$ and $\gamma\le \delta$ has countable cofinality then there is an $\w$-sequence cofinal in $\gamma$ and disjoint from $S$, so ``non-reflection'' at $\gamma$ is automatic.

\begin{proposition}\label{prop:main} Let $\delta$ be an ordinal and let $S\subseteq S^\delta_{\aleph_1}$. 
	\begin{enumerate}
	 	\item \label{item:absoluteness}
		Suppose that $V\subseteq W$ are transitive models of set theory, and that $\aleph_1^V = \aleph_1^W$. Then membership in $\PPP_\delta(S)$ is upward absolute between $V$ and $W$: $\left(\PPP_\delta(S)\right)^V\subseteq \left(\PPP_\delta(S)\right)^W$. 
		\item \label{item:size_of_forcing}
		If $\delta$ is an inaccessible cardinal, then for all $\bar \PP\in \PPP_\delta(S)$:
		\begin{enumerate}
			\item \label{itemm:size_of_forcing}
			$\PP_{\delta}\subset H_\delta$. 
			\item \label{itemm:size_of_continuum}
			In $V^{\PP_\delta}$, $\delta = 2^{\aleph_0}$. 
		\end{enumerate} 
		\item \label{item:continuity}
		If $\delta$ is an inaccessible cardinal	and $\funnyDiamond(S)$ holds, then there is some $\bar \PP\in \PPP_\delta(S)$ such that in $V^{\PP_\delta}$, every function $f\colon 2^\w\to 2^\w$ is continuous on a non-meagre set. 
		\item \label{item:Cohen}
		If $\alpha<\delta$, $\alpha\notin S$ and $S$ reflects no-where in the interval $(\alpha,\delta]$, Then for all $\bar \PP\in \PPP_\delta(S)$, $\PP_\delta/\PP_\alpha$ is equivalent to a Cohen algebra. 
		\item \label{item:preparation}
		If $\kappa$ is a measurable cardinal and $2^\kappa = \kappa^+$, then there is a forcing extension $W$ of $V$, preserving the measurability of $\kappa$, in which there is a stationary subset $S\subseteq S^\kappa_{\aleph_1}$ and a normal ultrafilter embedding $j\colon W\to N$ with critical point $\kappa$ such that:
		\begin{enumerate}
			\item $\funnyDiamond(S)$ holds; and
			\item \label{itemm:no_reflection_home} in $W$, $j(S)$ reflects no-where in the interval $(\kappa,j(\kappa)]$. 
		\end{enumerate}
	\end{enumerate}
\end{proposition}

Note that \eqref{itemm:no_reflection_home} means that for $\delta\in (\kappa,j(\kappa)]$ of uncountable cofinality, there is, in $W$, a club of $\delta$ disjoint from $S$. Such a club will often not exist in $N$. 

\medskip

Theorem~\ref{thm:main} is then proved as follows. Obtain a forcing extension $W$ given by~\eqref{item:preparation} of the proposition, and work in $W$. Pick an iteration $\bar \PP\in \PPP_\kappa(S)$ given by~\eqref{item:continuity}. The desired model is $W^{\PP_\kappa}$. By~\eqref{itemm:size_of_continuum}, in $W^{\PP_\kappa}$, $\kappa = 2^{\aleph_0}$; and by~\eqref{item:continuity}, in $W^{\PP_\kappa}$, every function $f\colon 2^\w\to 2^\w$ is continuous on a non-meagre set. 

Since $\PP_\kappa\subset H_\kappa$ \eqref{itemm:size_of_forcing}, $j\rest{\PP_\kappa}$ is the identity on $\PP_\kappa$ and the iteration $j(\bar \PP)$ is an extension of the iteration $\bar \PP$, so for $\alpha\le j(\kappa)$ we write $\PP_\alpha$ for $j(\PP)_\alpha$; we note that $j(\PP_\kappa) = \PP_{j(\kappa)}$. We can conclude that $\PP_\kappa\compembed \PP_{j(\kappa)}$. Now in $N$, $j(\bar \PP)\in \PPP_{j(\kappa)}(j(S))$, so by~\eqref{item:absoluteness}, $j(\bar \PP)\in \PPP_{j(\kappa)}(j(S))$ in $W$ as well. Since $\kappa$ is regular in~$W$ and $j(S)\subseteq S^{j(\kappa)}_{\aleph_1}$, $\kappa\notin j(S)$. Since~$j(S)$ reflects no-where in the interval $(\kappa,j(\kappa)]$, by applying~\eqref{item:Cohen} in~$W$ to~$\kappa$, $j(\kappa)$, $j(S)$ and $j(\bar \PP)$, we see that in $W^{\PP_{\kappa}}$, $\PP_{j(\kappa)}/\PP_\kappa$ is equivalent to a Cohen algebra. However, as is well known (but see Proposition~\ref{prop:jP_is_what_we_need} for completeness), because $j$ is a normal ultrafilter embedding, in $W^{\PP_\kappa}$ there is a $\kappa$-complete ideal~$I$ such that $\Powerset(\kappa)/I$ is isomorphic to the completion of $\PP_{j(\kappa)}/\PP_\kappa$, and hence to a Cohen algebra. Thus in $W^{\PP_\kappa}$, $\kappa$ is Cohen measurable. This completes the proof of Theorem~\ref{thm:main}.

\subsection{Structure of the paper}

In Section~\ref{sec:preliminaries} we settle notation, give basic definitions and recall some facts about forcing iterations, equivalence to Cohen algebras, and the forcing from~\cite{Sh473}, which as we mentioned will be one of the important ingredients of this paper. 

In Section~\ref{sec:restricted_iterations} we define a broad class of forcing iterations, from which elements of the collections $\PPP_\delta(S)$ will be taken. These are the iterations with ``restricted memory''. 

In Section~\ref{sec:S-iterations} we define the classes $\PPP_\delta(S)$. We easily observe that~\eqref{item:absoluteness} and~\eqref{item:size_of_forcing} of Proposition~\ref{prop:main} hold. In this section we also prove~\eqref{item:Cohen}.

In Section~\ref{sec:toward_measurability} we deviate a little from the proof of Theorem~\ref{thm:main}. Apart from giving a proof of the isomorphism of $j(\PP)/\PP$ and $\Powerset(\kappa)/I$ (Proposition~\ref{prop:jP_is_what_we_need}), which is mostly given for completeness, we show how to obtain the conclusion of Theorem~\ref{thm:main} if we are willing to start with a supercompact cardinal, equipped with a suitable stationary subset. In this case we do not need the preparation forcing which gives us~\eqref{item:preparation}. 

In Section~\ref{sec:continuity} we Define $\funnyDiamond(S)$ and prove~\eqref{item:continuity} of Proposition~\ref{prop:main}. 

In Section~\ref{sec:consistency} we prove~\eqref{item:preparation} of Proposition~\ref{prop:main}.

\section{Preliminaries}	 \label{sec:preliminaries}

We fix some notation and recall some basics.

\subsection{Complete embeddings}

Let $\PP\subseteq \QQ$ be partial orderings. A \emph{restriction} of~$\QQ$ to~$\PP$ is a function $i\colon \QQ\to \PP$ such that: (1) $i$ is order-preserving; (2) for all $q\in \QQ$, $q\le i(q)$; (3)~$i\rest{\PP} = \id_{\PP}$; and (4) for all $q\in \QQ$, every $p\le i(q)$ in $\PP$ is compatible with~$q$ in~$\QQ$. 

Note that if there is a restriction of~$\QQ$ to~$\PP$ then for all $p,q\in \PP$, $p\perp_\PP q$ if and only if $p\perp_\QQ p$; and every dense set $D\subseteq \PP$ is pre-dense in $\QQ$. In this paper we write $\PP\compembed \QQ$ if there is a restriction from $\QQ$ to $\PP$. This is equivalent to the usual notion in case $\PP$ and $\QQ$ are complete Boolean algebras and $\PP$ is a sub-algebra of $\QQ$.

If $\PP\subseteq \QQ$ then we let $\QQ/\PP$ be the $\PP$-name for the sub-ordering of $\QQ$ on $\QQ/G_\PP = \left\{ q\in \QQ  \,:\, \text{for all }p\in G_\PP,\, p\not\perp_{\QQ} q \right\}$. For $q\in \QQ$ and $p\in \PP$, $p\force q\in \QQ/\PP$ if and only if every $r\le p$ in $\PP$ is compatible with $q$ in $\QQ$. Thus, if $i$ is a restriction of $\QQ$ to $\PP$ then for all $q\in \QQ$, $i(q)\force_\PP q\in \QQ/\PP$. This of course holds for $p = i(q)$.  Also note that $p\force_\PP q\notin \QQ/\PP$ if and only if $p\perp_\QQ q$. Further, for $p\in\PP$ and $q\in \QQ$, if $p\force_\PP q\in \QQ/\PP$ then $p\force_\PP i(q)\in G_\PP$ (every $p'\le p$ is compatible with $q$ in $\QQ$, and so (by applying~$i$) compatible with $i(q)$ in $\PP$.) Thus $i(q)$ is essentially the weakest condition forcing that $q\in \QQ/\PP$ (only lack of separativity could cause it to not technically be the greatest such condition). Thus, if $G\subset \PP$ is generic, then for all $q\in \QQ$, $q\in \QQ/G$ if and only if $i(q)\in G$.

\begin{fact}\label{fact:generic_compatibility_in_quotient}
	Let $i\colon \QQ\to \PP$ be a restriction. If $D\subseteq \QQ$ is dense, then in $V^{\PP}$, $D\cap (\QQ/\PP)$ is dense in $\QQ/\PP$. In particular, for every $g\in G_\PP$ and every $q\in \QQ/\PP = \QQ/G_\PP$ there is some $\bar q\le q,g$ in $\QQ/\PP$. 
\end{fact}

\begin{fact} \label{fact:3_in_a_row}
	Suppose that $\PP\compembed \QQ \compembed \RR$; let $i$ be a restriction of $\QQ$ to $\PP$ and $j$ be a restriction of $\RR$ to $\QQ$. Then $i\circ j$ is a restriction of $\RR$ to $\PP$. In $V^{\PP}$, $\QQ/\PP \compembed \RR/\PP$. 
\end{fact}

For the following fact, recall that a map $i\colon \QQ\to \PP$ is called a \emph{dense homomorphism} if it preserves order and incompatibility, and its range is a dense subset of $\PP$. If there is such a map, then $\PP$ and $\QQ$ are forcing-equivalent.  

\begin{fact}\label{fact:dense_then_complete}
	Let $\PP\subseteq \QQ$, and suppose that $i\colon \QQ\to \PP$ is a dense homomorphism. Suppose that~$i$ is an idempotent: $i\rest \PP = \id_\PP$. Then~$i$ is a restriction of~$\QQ$ to~$\PP$.
\end{fact}

\begin{fact}\label{fact:dense_then_isomorphic_quotient}
	Let $\PP\subseteq \QQ$, and suppose that $i\colon \QQ\to \PP$ is dense. If $\QQ\compembed \RR$ then in $V^\QQ = V^\PP$, $\RR/\QQ = \RR/\PP$. 
\end{fact}

\subsection{Embeddings into a Cohen algebra}

For a set $X$, we let $C(X)$ be the finite support product, indexed by $X$, of one-dimensional Cohen forcing $C = (2^{<\w},\preceq)$. We let $\Cohen(X)$ be the completion of $C_X$ (the complete Boolean algebra of which $C_X$ is a dense subset). For disjoint sets $X,Y$ we write $\Cohen(X,Y)$ for $\Cohen(X\cup Y)$. We let $\Cohen = \Cohen(1)$. 

We say that a partial ordering (a notion of forcing) $\PP$ is \emph{equivalent to a Cohen algebra} if there is a dense embedding of $\PP$ into $\Cohen(X)$ for some $X$. We write $\PP\sim~\Cohen(X)$. 

If $X\subseteq Y$ then there is a natural embedding of $C(X)$ into $C(Y)$, which is complete. It induces a complete embedding of $\Cohen(X)$ into $\Cohen(Y)$. 

We will make use of the following Lemma. It is surely known; we include a proof for completeness. 

\begin{lemma}\label{lem:Cohen_extensions}
	Let $\PP\compembed \QQ$, and suppose that $\PP\sim \Cohen(X)$. Let $Y$ be a set. The following are equivalent:
	\begin{enumerate}
			\item $\force_\PP \QQ/\PP\sim\Cohen(Y)$;
			\item Every dense embedding of $\PP$ into $\Cohen(X)$ can be extended to a dense embedding of $\QQ$ into $\Cohen(X,Y)$. 
			\item There is a dense embedding of $\PP$ into $\Cohen(X)$ which can be extended to a dense embedding of $\QQ$ into $\Cohen(X,Y)$.
	\end{enumerate}
\end{lemma}

We will be imprecise and write ``$\QQ/\PP$ is (equivalent to) a Cohen algebra'' when we mean that for some $Y\in V$, 
\[ \force_\PP ``\QQ/G_\PP\text{ is equivalent to }\Cohen(Y)". \]

\begin{proof}
	In this proof, for neatness, let $\R = 2^\w$ be Cantor space. 
	
	Assume (1), we show (2). Let $j$ be a $\PP$-name for a dense embedding of $\QQ/\PP$ into $\Cohen(Y)$. 
	Let $q\mapsto q\rest{\PP}$ be a restriction map. For $q\in \QQ$ and $p\le q\rest{\PP}$, let 
	\[ U(p,q) = \bigcup D \Cyl{D\subseteq \R^Y \text{ is clopen, and }p\force_\PP D\subseteq^* j(q)},\] 	
	where we think of the elements of $\Cohen(Y)$ as regular open subsets of $\R^Y$, and $A\subseteq^* B$ means that $A\setminus B$ is meagre.

	 Let $i\colon \PP\to\Cohen(X)$ be a dense embedding. For $q\in \QQ$, let 
	\[ k(q) = \bigcup_{p\le q\rest{\PP}} i(p)\times U(p,q) .\]
	
	Let $q_0,q_1\in \QQ$. Say $q_1 \le q_0$. Then $q_1\rest{\PP}\le q_0\rest{\PP}$, and for all $p\le q_1\rest{\PP}$, $U(p,q_1)\subseteq U(p,q_0)$. Hence $k(q_1) \subseteq k(q_0)$. 
	
	Suppose that $k(q_0)$ and $k(q_1)$ are compatible; let $E\subseteq \R^X$ and $D\subseteq \R^Y$ be clopen such that $E\times D \subseteq k(q_0) \cap k(q_1)$. First find $p_0\le q_0\rest{\PP}$ such that $E\cap i(p_0)\ne \emptyset$ and $D\cap U(p_0,q_0)\ne \emptyset$. Find $E_0$ and $D_0$ clopen such that $E_0\times D_0 \subseteq (E\times D)\cap (i(p_0)\times U(p_0,q_0))$. Find $p_1\le q_1\rest{\PP}$ such that $E_0\cap i(p_1)\ne \emptyset$ and $D_0\cap U(p_1,q_1)\ne \emptyset$. So $p_0$ and $p_1$ are compatible in $\PP$; let $p\in \PP$ extend both; $p$ forces that $j(q_0)$ and $j(q_1)$ are compatible, so $p$ forces that $q_0$ and $q_1$ are compatible in $\QQ/G_\PP$, and so $q_0$ and $q_1$ are compatible. 
	
	Let $E\times D$ be clopen. There is some $p\in \PP$ such that $i(p) \subseteq^* E$. There is some $\bar p\le_\PP p$ and some $q\in \QQ$ such that $\bar p\force q\in \QQ/G_\PP \andd j(q)\subseteq^* D$. In particular, $\bar p\not\perp_\QQ q$; let $\bar q\le_\QQ \bar p, q$. Then $\bar q\rest{\PP}\le \bar p\rest{\PP} = \bar p$. For any $r\le \bar q\rest {\PP}$ we have $i(r)\subseteq^* E$ and $r\force_\PP j(\bar q)\subseteq^* D$, so $U(r,\bar q)\subseteq^* D$. It follows that $k(\bar q)\subseteq^* E\times D$.

Let $p\in \PP$. Then $p$ forces that $p$ is compatible with every $q\in \QQ/G_\PP$. It follows that $U(p,p) = \R^Y$. For let $D\subseteq \R^Y$ be clopen. In $\PP$, densely below $p$ we can find $p'$ for which we can find some $q'\in \QQ$ with $p'\force q'\in \QQ/G_\PP$ and $p'\force j(q')\subseteq^* D$. Since $p'$ also forces that $p$ and $q'$ are compatible, it forces that $j(p)\cap D$ is nonempty. Hence $p$ forces that $j(p)$ is dense, i.e.\ that $j(p) =^* \R^Y$. Hence $k(p) = i(p)\times \R^Y$.

\

Assume (3), we show (1). Let $i\colon \PP\to \Cohen(X)$ be a dense embedding, and let $k\colon \QQ\to \Cohen(X,Y)$ extend $i$. Let $G$ be $\PP$-generic, and let $\bar r^G$ be the Cohen generic sequence in $\R^X$. In $V[G]$, for $q\in \QQ/G$ let $j(q)$ be the section $k(q)_{\bar r^G}$ which is an open subset of $\R^Y$. 

Let $q_0,q_1\in \QQ/G$. If $q_1\le_\QQ q_0$ then $k(q_1)\subseteq k(q_0)$ (we use maximal representatives, i.e., regular open sets) and so $j(q_0)\subseteq j(q_1)$. 

We show that if $p\force q\in \QQ/G$ and $p\force D\subseteq j(q)$ (for some clopen $D$) then $i(p)\times D\subseteq^* k(q)$. Suppose not; find some clopen $E$ and $C$ such that $E\times C \subseteq i(p)\times D$ but $E\times C$ is disjoint from $k(q)$. Find some $\bar p\le p$ such that $i(\bar p)\subseteq E$. Since $\bar p$ forces that $\bar r^G\in E$, it forces that $C$ is disjoint from $j(q)$, which is impossible.

Suppose that $j(q_0)$ and $j(q_1)$ are compatible; let $p\in G$ force that $D\subseteq j(q_0)\cap j(q_1)$ for some nonempty clopen subset $D$ of $\R^Y$, and that $q_0,q_1\in \QQ/G$. Then $i(p)\times D\subseteq k(q_0)\cap k(q_1)$. For densely many $\bar p\le p$ (in $\PP$) there is some $q\le_\QQ q_0, q_1$ such that $\bar p \force q\in \QQ/G$. For let $p'\le p$. Let $q\le_{\QQ}q_0,q_1$ such that $k(q)\cap (i(p')\times D)$ is nonempty. Let $\bar p\le p'$ such that $i(\bar p)\times D'\subseteq k(q)$ for some nonempty clopen $D'\subseteq D$. Then $\bar p\force q\in \QQ/G$.  

Let $D\subseteq \R^Y$ be clopen. Given $p\in \PP$, find some $q\in \QQ$ such that $k(q) \subseteq i(p)\times D$. Find some $\bar p\le p$ and some $D'\subseteq D$ such that $i(\bar p)\times D'\subseteq k(q)$. So $\bar p\force q\in \QQ/G$ and $j(q)\subseteq D$. 
\end{proof}


\subsection{A restricted form for iterations}

All partial orderings have a greatest element, usually denoted by $1$. 

\medskip

We restrict ourselves to two-step iterations of the following form: $\PP$ is a partial ordering, $(R,\le)$ is some partial ordering, and $S$ is a $\PP$-name for a non-empty upward-closed subset of $R$ (in particular, $\force_\PP 1_R\in S$); we assume that as a name, $S\subseteq \PP\times R$. We then let 
$\PP\starr S$ be the collection of pairs $(p,s)\in \PP\times R$ such that $p\force s\in S$, ordered as a sub-ordering of $\PP\times R$. We note that if $\PP\compembed \QQ$ (with restriction $i$), then $\QQ/\PP$ is a $\PP$-name for an upward-closed subset of $\QQ$. The map $q\mapsto (i(q),q)$ is a dense embedding of $\QQ$ into $\PP\starr (\QQ/\PP)$, so these notions of forcing are equivalent.

\begin{fact}\label{fact:two_step_trivial}
	Suppose that $\PP\compembed \QQ$ with $i\colon \QQ\to \PP$ a restriction. Let $S\subseteq \QQ\times R$ be a $\QQ$-name for an upward-closed subset of a partial ordering $R$. Then the map $(q,s)\mapsto i(q)$ is a restriction of $\QQ\starr S$ to $\PP$.
\end{fact}

And so $\PP\compembed \QQ\starr S$. Here we identify $\PP$ with its image in $\QQ\starr S$ under the map $p\mapsto (p,1)$. In particular, of course, $\QQ\compembed \QQ\starr S$.

\medskip

For the following note that if $\PP\compembed \QQ$ and $S\subseteq \PP\times R$ is a~$\PP$-name for a subset of~$R$, then~$S$ is also a~$\QQ$-name for a subset of $R$. 

\begin{fact}\label{fact:two-step}
	Suppose that $\PP\compembed \QQ$ with $i\colon \QQ\to \PP$ a restriction. Let $S\subseteq \PP\times R$ be a $\PP$-name for an upward-closed subset of a partial ordering $R$. Then:
	\begin{enumerate}
		\item For all $q\in \QQ$ and $s\in R$, $q\force_{\QQ} s\in S$ if and only if $i(q)\force_{\PP} s\in S$. In particular, for $p\in \PP$, $p\force_\PP s\in S$ if and only if $p\force_\QQ s\in S$. 
		\item $S$ is also a $\QQ$-name for an upward-closed subset of $R$. 
		\item The map $(q,s)\mapsto (i(q),s)$ is a restriction of $\QQ\starr S$ to $\PP\starr S$.
	\end{enumerate}
\end{fact}

\subsection{Forcing continuity on a non-meagre set}

We fix notation for the notion of forcing from~\cite{Sh473}. In full generality, let $\PP$ be a notion of forcing, and let $\bar \eta = \seq{\eta_i}_{i<\w_1}$ and $\bar \zeta = \seq{\zeta_i}_{i<\w_1}$ be two sequences of $\PP$-names for reals (in this paper, elements of Cantor space $2^\w$). We let $\Shelah(\bar \eta,\bar \zeta)$ be the $\PP$-name for the notion of forcing which adds the definition of a continuous function which makes the map $\eta_{i}\mapsto \zeta_{i}$ continuous on a non-meagre set. Technically, the conditions in $\PP*\Shelah(\bar \eta,\bar \zeta)$ will be pairs $(p,a,\Psi)$, where:
\begin{itemize}
	\item $p\in \PP$;
	\item $a$ is a finite subset of $\w_1$;
	\item $\Psi$ is a finite, (strict) order-preserving map from $2^{<\w}$ to $2^{<\w}$, and:
	\begin{enumerate}
		\item $p$ forces that every element of $\dom \Psi$ is an initial segment of $\eta_i$ for some $i\in a$; and
		\item If $i\in a$, $(\s,\tau)\in \Psi$ and $p\force_\PP \s\prec \eta_i$ then $p\force_\PP \tau\prec \zeta_i$. 
	\end{enumerate}
\end{itemize}

A condition $(q,b,\Phi)$ extends a condition $(p,a,\Psi)$ if $q$ extends $p$ in $\PP$, $a\subseteq b$, and $\Psi\subseteq \Phi$. Note that this is an example of a two-step iteration which obeys the restrictions above: $\Shelah(\bar \eta,\bar \zeta)$ is a $\PP$-name for an upward closed subset of $R$ where $R$ consists of pairs $(a,\Psi)$ where $a$ is a finite subset of $\w_1$ and $\Psi$ is a finite, (strict) order-preserving map from $2^{<\w}$ to $2^{<\w}$, ordered by $\subseteq$ on both coordinates. Note that $R\subset H(\w_1)$ and $|R| = \aleph_1$. 

\medskip

Shelah's notion of forcing starts with $\PP = \Cohen(\w_1)$. Letting $\eta_i$ be the name for the Cohen real added by $\Cohen(\{2i\})$ and $\zeta_i$ be the name for the Cohen real added by $\Cohen(\{2i+1\})$, Shelah uses $\Cohen(\w_1)\starr \Shelah(\bar \eta,\bar \zeta)$. We denote this simply by $\Cohen(\w_1)\starr \Shelah$. 

\begin{proposition}[\cite{Sh473}]\label{prop:Shelah_properties} \
	\begin{enumerate}
		\item In $V^{\Cohen(\w_1)\starr \Shelah}$ there is a non-meagre subset of $\{\eta_{i}\,:\, i<\w_1\}$ on which the map $\eta_{i}\mapsto \zeta_{i}$ is continuous.
		\item For all $i<\w_1$, $\Cohen(2i)\compembed \left(\Cohen(\w_1)\starr \Shelah\right) $ and $\left(\Cohen(\w_1)\starr \Shelah\right)/\Cohen(2i)$ is equivalent to a Cohen algebra (of dimension $\aleph_1$). 
	\end{enumerate}
\end{proposition}

\section{Iterations with restricted memory} \label{sec:restricted_iterations}

Let $\bar \PP = \seq{\PP_\alpha,\QQ_\alpha}_{\alpha<\delta}$ be a finite support iteration. For $\alpha\le \delta$, we think of the elements of $\PP_\alpha$ as sequences of length $\alpha$. As above, we suppose that all successor steps are ``$V$-based'', in the sense that for all $\alpha<\delta$ there is some partial ordering $(R_\alpha,\le)$ in $V$ such that $\QQ_\alpha$ is a $\PP_{\alpha}$-name for an upward-closed subset of $R_\alpha$, and the ordering on $\PP_{\alpha+1} = \PP_{\alpha}\starr \QQ_\alpha$ is inherited from the one on $\PP_\alpha\times R_\alpha$. It follows that $\PP_\delta$ is a subset of the finite-support product $\bigoplus_{\alpha<\delta}R_\alpha$ of the $R_\alpha$'s, with the inherited ordering. That is, for $p,q\in \PP_\delta$, $p\le_{\PP_\delta} q$ if and only if for all $\alpha<\delta$, $p(\alpha)\le_{R_\alpha} q(\alpha)$. Below, we will always assume the existence of such ambient orderings $R_\alpha$.

For $u\subseteq \delta$ we let 
\[ \PP_u = \left\{ p\in \PP_\delta  \,:\, \text{for all }\alpha\in \delta\setminus u,\, p(\alpha)=1 \right\},\]

with order inherited from $\PP_\delta$. Technically we should have called this $\PP_{u,\delta}$. However, if $u\subseteq \alpha< \delta$ then $\PP_{u,\alpha}$ and $\PP_{u,\delta}$ are naturally isomorphic by appending a sequence of ones, so we ignore the difference between them. Under this identification there is no conflict between the two meanings of $\PP_{\alpha}$ for $\alpha<\delta$. Note that $\PP_u$ is upward-closed in $\PP_\delta$. If $u\subseteq v\subseteq \delta$ then $\PP_u\subseteq \PP_v$. 

\medskip

Now the main point is that usually, unless $u=\alpha$ is an initial segment of $\delta$, $\PP_u$ will not contain much. For example, if $0\notin u$, but to define each $\QQ_\alpha$ (for $\alpha>0$) we need access to $\QQ_0$, then $\PP_u$ will contain very little, since each condition in $\PP_u$ knows nothing about $\QQ_0$. In the other extreme, if the iteration is actually a product, no~$\QQ_\alpha$ needs any information about any other~$\QQ_\beta$, and in this case, for any $u\subseteq \delta$,~$\PP_u$ is just the product restricted to~$u$, which behaves perfectly nicely; in particular, $\PP_u\compembed \PP_\delta$. As we mentioned in the introduction, we will be using iterations which are not quite products but for which each~$\QQ_\alpha$ needs information from ``not so many''~$\QQ_\beta$ for $\beta<\alpha$, and so for many sets $u\subseteq \delta$ we will have $\PP_u\compembed \PP_\delta$. In other words, each~$\QQ_\beta$ will have ``restricted memory''. A~\emph{memory template} for the iteration~$\bar \PP$ specifies, for each~$\alpha$, which~$\QQ_\beta$ (for~$\beta<\alpha$) are needed to compute~$\QQ_\alpha$. 

\begin{definition}\label{def:memory_template}
A \emph{memory template} (of length $\delta$) is a sequence $\uu = \seq{\uu_\alpha}_{\alpha<\delta}$ such that for all $\alpha<\delta$, 
\begin{itemize}
	\item $\uu_\alpha\subseteq \alpha$; and
	\item if $\beta\in \uu_\alpha$ then $\uu_\beta\subset \uu_\alpha$. 
\end{itemize}
\end{definition}

The second condition is a natural transitivity requirement: if $\QQ_\beta$ is needed to compute $\QQ_\alpha$, and $\QQ_\gamma$ is needed to compute $\QQ_\beta$, then certainly $\QQ_\gamma$ is needed to compute $\QQ_\alpha$. 
Note that if $\uu$ is a memory template of length $\delta$ and $\alpha<\delta$, then $\uu\rest \alpha$ is a memory template of length $\alpha$. 

\begin{definition}\label{def:u-iteration}
Let $\uu$ be a memory template of length $\delta$. A finite support iteration $\bar \PP = \seq{\PP_\alpha,\QQ_\alpha}_{\alpha<\delta}$ is a \emph{$\uu$-iteration} if for all $\alpha<\delta$, $\QQ_\alpha$ is a $\PP_{\uu_\alpha}$-name. (We will show below that in this case, $\PP_{\uu_\alpha}\compembed \PP_\delta$, which means that this definition makes more sense.)
\end{definition}

Here we assume as above that as a name, $\QQ_\alpha$ is a subset of $\PP_{\uu_\alpha}\times R_\alpha$ (where $R_\alpha$ is the ambient partial ordering for $\QQ_\alpha$ mentioned above).

\begin{definition}\label{def:u-closure}
Let $\uu$ be a memory template of length $\delta$. A subset $u$ of $\delta$ is \emph{$\uu$-closed} if for all $\alpha\in u$, $\uu_\alpha\subset u$. 
\end{definition}

So each $\uu_\alpha$ is $\uu$-closed. Note that each $\alpha<\delta$ is $\uu$-closed, indeed if $u\subseteq \delta$ is $\uu$-closed and $\alpha<\delta$ then $u\cap \alpha$ is $\uu$-closed. A subset of $\alpha<\delta$ is $\uu$-closed if and only if it is $\uu\rest \alpha$-closed. 

\medskip

Let $\bar \PP$ be an iteration of length $\delta$, and let $u\subseteq \delta$. For $p\in \PP_\delta$ we define a $\delta$-sequence $p\rest u$ by letting, for $\alpha<\delta$, 
\[ p\rest{u} (\alpha) = \begin{cases}
p(\alpha),\quad & \text{if }\alpha\in u;\\
1,\quad & \text{if }\alpha\notin u.
\end{cases}\]
Note that for all $p\in \PP_u$, $p\rest{u} = p$. In general, $p\rest u$ may not be an element of $\PP_\delta$, again, because erasing part of its head may cause us to lose the evidence for its tail being in $\PP_\delta$.

\begin{lemma}\label{lem:the_restriction_function}
Let $\uu$ be a memory template of length $\delta$, let $\bar \PP$ be a $\uu$-iteration, and let $u\subseteq v\subseteq \delta$ be $\uu$-closed. Then for all $q\in \PP_v$, $q\rest u \in \PP_u$, and the map $q\mapsto q\rest u$ is a restriction of $\PP_v$ to $\PP_u$ (so $\PP_u\compembed \PP_v$).
\end{lemma}

\begin{proof}
By induction on $\delta$. First, suppose that $\delta$ is a limit ordinal and that the lemma holds for all $\alpha<\delta$. That the lemma holds for $\delta$ follows from the fact that $\bar \PP$ is of finite support, so viewed as sequences of length below $\delta$, we have $\PP_{v} = \bigcup_{\alpha<\delta}\PP_{v\cap \alpha}$, and $\PP_{u} = \bigcup_{\alpha<\delta} \PP_{u\cap \alpha}$; and noting again that for all $\alpha<\delta$, $u\cap \alpha$ and $v\cap \alpha$ are $\uu\rest \alpha$-closed subsets of $\alpha$. 

Next let $\delta$ be any ordinal and suppose that the lemma holds for $\delta$; we show it holds for $\delta+1$. 
Of course the point is that if $w\subseteq \delta+1$ is $\uu$-closed and $\delta\in w$, then $\uu_\delta\subset w$. Then $w\cap \delta$ is $\uu\rest{\delta}$-closed, and by induction and by Fact~\ref{fact:two-step}, $\QQ_\delta\subseteq \PP_{w\cap \delta}\times R_\delta$ is also a $\PP_{w\cap \delta}$-name for an upward-closed subset of $R_\delta$. 

Let $u\subseteq v\subseteq \delta+1$ be $\uu$-closed. There are three cases:
\begin{itemize}
\item If $\delta\notin v$ then the lemma for $u$ and $v$ 
follows from the fact it holds at stage~$\delta$.
\item If $\delta \in v$ but $\delta\notin u$ then the lemma for~$u$ and~$v$ follows from Fact~\ref{fact:two_step_trivial} applied to $\PP = \PP_{u}$ and $\QQ= \PP_{v\cap \delta}$.
\item If $\delta\in u$ then the lemma for $u$ and $v$ follows from Fact~\ref{fact:two-step} applied to $\PP = \PP_{u\cap \delta}$ and $\QQ= \PP_{v\cap \delta}$. \qedhere
\end{itemize}
\end{proof}

For the rest of this section, let $\uu$ be a memory template (of length $\delta$) and let $\bar \PP$ be a $\uu$-iteration.

\begin{porism}\label{por:obvious_identity}
	Let $u\subseteq \delta$ be $\uu$-closed and suppose that $u$ has a greatest element $\alpha$. Then $\PP_u = \PP_{u\cap \alpha}\starr \QQ_\alpha$.
\end{porism}

\medskip

Let $p,q\in \bigoplus_\alpha R_\alpha$ and suppose that for all $\alpha<\delta$, $p(\alpha)$ and $q(\alpha)$ are comparable in $R_\alpha$. Then we can define $p\meet q = \min(p,q) \in \bigoplus R_\alpha$ by taking at every $\alpha$ the smaller of the two values $p(\alpha)$ and $q(\alpha)$;  $p\meet q$ is the greatest lower bound of $p$ and $q$ in $\bigoplus R_\alpha$. An induction on $\delta$ shows that if $p,q\in \PP_\delta$ then $p\meet q\in \PP_\delta$ as well, and so $p\meet q$ is the greatest lower bound of $p$ and $q$ in $\PP_\delta$, similarly in $\PP_u$ for any $\uu$-closed $u$ such that $p,q\in \PP_u$. 

In particular, in the situation above ($u\subseteq v\subseteq \delta$ are $\uu$-closed), if $q\in \PP_v$, $p\in \PP_u$ and $p\le q\rest u$, then $p\meet q\in \PP_v$ is the greatest lower bound of $p$ and $q$ in $\PP_v$.

\medskip

For the next lemma, note that if $u,v\subseteq \delta$ are $\uu$-closed, then $u\cap v$ and $u\cup v$ are also $\uu$-closed. The map from $\PP_{u\cup v}$ to $\PP_u\times \PP_v$ given by $q\mapsto (q\rest u,q\rest v)$ is injective, and preserves both order and non-order. 

\begin{lemma}\label{lem:tensor_product}
Let $u,v\subseteq \delta$ be $\uu$-closed. In $V^{\PP_{u\cap v}}$, the map $q\mapsto (q\rest u,q\rest v)$ is a dense embedding of $\PP_{u\cup v}/\PP_{u\cap v}$ into $(\PP_u/\PP_{u\cap v})\times (\PP_v/\PP_{u\cap v})$. 
\end{lemma}

\begin{proof}
Let $H\subseteq \PP_{u\cap v}$ be generic; work in $V[H]$. Certainly if $q\in \PP_{u\cup v}/H$ then as $q\rest{u\cup v}\in H$ 
and $(q\rest{u})\rest{u\cap v} = q\rest{u\cap v}$ (and similarly for $v$), 
again we have $q\rest u \in \PP_u/H$ and $q\rest v \in \PP_v/H$. So the map $j(q) = (q\rest u,q\rest v)$ is indeed from $\PP_{u\cup v}/H$ to $(\PP_u/H) \times (\PP_v/H)$. 

We show that the range of $j$ is dense. Let $(p,q)\in (\PP_u/H)\times (\PP_v/H)$. Then $p\rest{u\cap v}, q\rest{u\cap v}\in H$; find some $g\in H$ extending both. From $g\le p\rest{u\cap v}$ we conclude that $(g\meet p)\in \PP_u/H$. Now $q\rest u = q\rest{u\cap v}\ge g \ge (g\meet p)$, and so $r = (g\meet p)\meet q\in \PP_v/H$. Then $j(r) = (g\meet p,g\meet q)\le (p,q)$. 

We show that $j$ preserves incompatibility. Let $r_0,r_1\in \PP_{u\cup v}/H$, and suppose that $j(r_0)$ and $j(r_1)$ are compatible in $(\PP_u/H)\times (\PP_v/H)$; let $h \le (j(q_0),j(q_1))$ witness this, and then find some $s\in \PP_{u\cup v}/H$ with $j(s)\le h$. So $s\rest u\le q_0\rest u,q_1\rest u$ and $s\rest v \le q_0\rest v,q_1\rest v$, and we conclude that $s\le q_0,q_1$. 
\end{proof}

Thus, the filters $K\subset \PP_{u\cup v}$ generic over $V$ correspond to the filters $G\subset \PP_u$ and $H\subset \PP_v$ with $G$ generic over $V$, $H$ generic over $V[G]$ and $G\cap \PP_{u\cap v} = H\cap \PP_{u\cap v}$, with $K = \{ g\meet h\,:\, g\in G \andd h\in H\}$. 

\begin{porism} \label{por:compatibility_in_product}
	Let $u,v\subseteq \delta$ be $\uu$-closed and let $H\subseteq \PP_{u\cap v}$ be generic. If $(p,q)\in (\PP_u/H) \times (\PP_v/H)$ then there is some $h\in H$ such that $(h\meet p)\meet q \in \PP_{u\cup v}$. 
\end{porism}

\begin{lemma}\label{lem:tensor_2}
Let $u,v\subseteq \delta$ be $\uu$-closed. In $V^{\PP_u}$, the map $q\mapsto q\rest{v}$ is a dense embedding of $\PP_{u\cup v}/\PP_u$ into $\PP_{v}/\PP_{u\cap v}$. 
\end{lemma}

\begin{proof}
Let $G\subseteq \PP_u$ be generic, and let $H = G\cap \PP_{u\cap v}$. Since $\PP_{u\cup v}/G \subseteq \PP_{u\cup v}/H$, and we noticed that if $q\in \PP_{u\cup v}/H$ then $q\rest v \in \PP_v/H$, the map $i$ defined on $\PP_{u\cup v}/G$ defined by $q\mapsto q\rest v$ is indeed into $\PP_v/H$. It is order-preserving.

The map $i$ is onto $\PP_v/H$. For let $p\in \PP_v/H$; so $p\rest u = p\rest{u\cap v} \in H\subseteq G$, so $p\in \PP_{u\cup v}/G$, and $p = i(p)$. It remains to show that $i$ preserves incompatibility. 

Let $j\colon \PP_{u\cup v}/H\to (\PP_u/H)\times (\PP_v/H)$ be the dense embedding $q\mapsto (q\rest u,q\rest v)$. For $q\in \PP_{u\cup v}/H$, we have $q\rest u\in G$ if and only if $q\in \PP_{u\cup v}/G$ and so $j^{-1}[G\times (\PP_v/H)] = \PP_{u\cup v}/G$. 

Let $r_0,r_1\in \PP_{u\cup v}/G$, and suppose that $i(q_0)$ and $i(q_1)$ are compatible in $\PP_v/H$.  Since $r_0\rest u,r_1\rest u\in G$, it follows that $j(r_0)$ and $j(r_1)$ are compatible in $G\times (\PP_v/H)$; let $(g,p)\in G\times (\PP_v/H)$ extend both $j(r_0)$ and $j(r_1)$. By Porism~\ref{por:compatibility_in_product} there is some $h\in H$ such that $r=(h\meet g)\meet p\in \PP_{u\cup v}$. Since $(h\meet g)\in G$, $r\in \PP_{u\cup v}/G$ extends both $r_0$ and $r_1$. 
\end{proof}

\begin{corollary}\label{cor:below_an_inaccessible_relent}
Let $u\subseteq v\subseteq \delta$ be $\uu$-closed, let $\alpha<\delta$ and suppose that $u\cap [\alpha,\delta) = v\cap [\alpha,\delta)$. Then in $V^{\PP_u}$, $\PP_v/\PP_u$ is equivalent to $\PP_{v\cap \alpha}/\PP_{u\cap \alpha}$. 
\end{corollary}

\begin{proof}
Immediate from Lemma~\ref{lem:tensor_2}, since $v = u\cup (v\cap \alpha)$ and $u\cap \alpha = u \cap (v\cap \alpha)$. 
\end{proof}

\begin{corollary}\label{cor:2-dimensional_completeness}
Let $u\subseteq v\subseteq \delta$ be $\uu$-closed. Let $\beta<\delta$. Then in $V^{\PP_u}$, $\PP_{v\cap\beta}/\PP_{u\cap \beta}\compembed \PP_v/\PP_u$ and in $V^{\PP_{u\cup (v\cap \beta)}}$, 
\[ \frac{\PP_v/\PP_u}{\PP_{v\cap \beta}/\PP_{u\cap \beta}} = \frac{\PP_v}{\PP_{u\cup (v\cap \beta)}}.\]
\end{corollary}	

\begin{proof}
For neatness, let $\bar v = v\cap \beta$ and $\bar u = u\cap \beta = u\cap \bar v$. First, let $G\subset \PP_u$ be generic and $H = G\cap \PP_{\bar u}$. The previous lemmas show that in $V[G]$, $\PP_{\bar v}/H \compembed \PP_{\bar v\cup u}/G$ and $\PP_{\bar v}/H \sim \PP_{\bar v\cup u}/G$. Since $\PP_u \compembed \PP_{u\cup \bar v}\compembed \PP_v$, Fact~\ref{fact:3_in_a_row} shows that $\PP_{\bar v\cup u}/G\compembed \PP_v/G$, so overall, $\PP_{\bar v}/H \compembed \PP_{v}/G$. But also, Fact~\ref{fact:dense_then_isomorphic_quotient} shows that in $V[G]^{\PP{\bar v\cup u}}$, 
\[ \frac{\PP_v/G}{\PP_{\bar v\cup u}/G} = \frac{\PP_v/G}{\PP_{\bar v}/H}.\]
Let $K\subset \PP_{\bar v\cup \bar u}/G$ be generic over $V[G]$, and let $\bar K = K\cap \PP_{\bar v}$. Then 
\[ \frac{\PP_v/G}{\PP_{\bar v}/H}[K] = \PP_v / \bar K,\]
and
\[ \frac{\PP_v/G}{(\PP_{\bar v\cup u}/G)}[K] = \PP_v/K.\]
So $\PP_v / K = \PP_v /\bar K$, and this is what needs to be shown. 
\end{proof}

Finally, we show:

\begin{lemma}\label{lem:limit_of_quotients}
Suppose that $c\subseteq \delta$ is unbounded in $\delta$. Let $u\subseteq v\subseteq \delta$ be $\uu$-closed. Then in $V^{\PP_u}$, 
\[ \PP_v/\PP_u = \bigcup_{\beta\in c} \PP_{v\cap \beta} / \PP_{u\cap \beta}. \]
\end{lemma}

\begin{proof}
This is only interesting if $\delta$ is a limit ordinal. Let $G\subset \PP_u$ be generic; for $\beta\in c$, let $G_\beta = G\cap \PP_{u\cap \beta}$. We already know that $\PP_{v\cap \beta}/ G_\beta \subseteq \PP_v / G$. Let $p\in \PP_v / G$. Since $\PP_v = \bigcup_{\beta\in c} \PP_{v\cap \beta}$, for some $\beta\in c$ we have $p\in \PP_{v\cap \beta}$. Since $p\rest {u\cap \beta} = (p\rest u)\rest{u\cap \beta}$ and $p\rest \beta\in G$, we have $p\rest{u\cap \beta}\in G_\beta$ and so $p\in \PP_{v\cap \beta}/G_\beta$. 
\end{proof}

\section{$S$-iterations} \label{sec:S-iterations}

Fix an ordinal $\delta$ and a set $S\subseteq S^\delta_{\aleph_1}$. We wish to define the class $\PPP_\delta(S)$ of iterations of length $\delta$ which are ``mostly Cohen'' but on elements of $S$ are allowed to deviate from being precisely Cohen. 
%
%

\begin{definition}\label{def:S-memory_template}
	A {memory template $\uu$ of length $\delta$} is an \emph{$S$-memory template} if:
	\begin{enumerate}
		\item For all $\alpha\in \delta\setminus S$, $\uu_\alpha = \emptyset$;
		\item For all $\alpha$, $|\uu_\alpha|\le \aleph_1$; and
		\item For all $\alpha\in S$,
		every $\beta\in S\cap \alpha$ which is a limit point of $\uu_\alpha$ is an element of~$\uu_\alpha$. 
	\end{enumerate}
\end{definition}

\begin{definition}\label{def:S-iteration}
	Let $\uu$ be an $S$-memory template of length $\delta$. A~$\uu$-iteration $\bar \PP$ is a \emph{$\uu$-quasi-Cohen} iteration if:
	\begin{enumerate}
		\item For all $\alpha\in \delta\setminus S$, $\QQ_\alpha = \Cohen$ is 1-dimensional Cohen forcing;
		\item For all $\alpha<\delta$, $R_\alpha\subset H_{\w_1}$ and $|R_\alpha|\le \aleph_1$ (where recall that $R_\alpha$ is the ambient partial ordering from which $\QQ_\alpha$ is taken as a subset); and
		\item For all $\alpha\in S$, for all $\beta\in \alpha\setminus S$, 
		$\PP_{\uu_\alpha\cup \{\alpha\}} / \PP_{\uu_\alpha\cap \beta} $
		is equivalent to a Cohen algebra. 
	\end{enumerate}
	We let $\PPP_\delta(S)$ denote the set of iterations $\bar \PP$ of length $\delta$ which are $\uu$-quasi-Cohen iterations for some $S$-memory template $\uu$. 
\end{definition}
\medskip

Having defined $\PPP_\delta(S)$, we note that the definition is upward absolute provided $\aleph_1$ does not change; this gives us~\eqref{item:absoluteness} of Proposition~\ref{prop:main}. Let $\bar \PP\in \PPP_\delta(S)$. If $\delta$ is an inaccessible cardinal then the fact that $R_\alpha\subset H_{\w_1}$ for all $\alpha<\delta$ implies that $\PP_\delta\subset H_\delta$. Further, since in this case $|\delta\setminus S| = \delta$, and each $\QQ_\alpha$ for $\alpha\notin S$ adds a Cohen real, in $V^{\PP_\delta}$, $2^{\aleph_0}\ge \delta$. To show that $2^{\aleph_0}\le \delta$ in $V^{\PP_\delta}$ we prove a general lemma which will be useful later as well. 

\begin{lemma}\label{lem:get_it_down_to_aleph_1}
	Let $\delta$ be any ordinal and $S\subseteq S^\delta_{\aleph_1}$. Let $\bar \PP\in \PPP_\delta(S)$, and let $\eta$ be a $\PP_\delta$-name for a real. Then there is some $\uu$-closed set $u\subset \delta$ of size at most $\aleph_1$ such that $\eta$ is a $\PP_v$-name. 
\end{lemma}

\begin{proof}
	Let $\uu$ be an $S$-memory template which witnesses that $\bar\PP\in \PPP_\delta(S)$. Let $M\elementary V$ (you know what we mean) of size $\aleph_1$ such that $\w_1\subset M$, with $\delta,S,\uu,\bar \PP,\eta \in M$. Let $u = M\cap \delta$. First, observe that $u$ is $\uu$-closed. If $\alpha\in u$ then $\alpha\in M$ and so $\uu_\alpha\in M$. Since $\w_1\subset M$ and $|\uu_\alpha| \le \aleph_1$, $\uu_\alpha \subset M$, so $\uu_\alpha\subset u$. 

	We claim that $\PP_u = \PP_\delta \cap M$. In one direction, let $p\in \PP_\delta\cap M$. Then the support $\supp(p)$ of $p$ is finite and is an element of $M$, and so $\supp(p)\subset M$. In the other direction, let $p\in \PP_u$. Since $\supp(p)$ is finite and is a subset of $M$, it is an element of $M$. For each $\alpha\in u$, $R_\alpha\in M$ and since $|R_\alpha|\le \aleph_1$, $R_\alpha \subset M$, and so, for all $\alpha \in \supp(p)$, $p(\alpha)\in M$. It follows that $p\rest {\supp(\alpha)}$ (restriction as a function) is in $M$, whence $p\in M$. 

	For $n<\w$, let $A_n$ be the set of conditions $p\in \PP_\delta$ such that $p\force_{\PP_\delta} \s\prec \eta$ for some $\s\in 2^n$. Then $A_n$ is dense in $\PP_\delta$, and so $A_n\cap M$ is dense in $\PP_\delta\cap M$. Since $\PP_u\compembed \PP_\delta$, we see that $A_n\cap \PP_u$ is dense in $\PP_\delta$, whence $\eta$ is a $\PP_u$-name. 
\end{proof}

\medskip

In order to show that~\eqref{item:Cohen} of Proposition~\ref{prop:main} holds, we prove something stronger, which is necessary elsewhere but also for the inductive proof. 

\begin{definition}\label{def:straight}
	Let $\uu$ be an $S$-memory template of length $\delta$.
	\begin{enumerate}
		\item Let $u\subseteq v\subseteq \delta$ be $\uu$-closed. We say that $v$ is a \emph{$\uu$-straight extension} of $u$ if for all $\beta\in v\cap S$, if $u\cap \beta$ is unbounded in $\beta$ then $\beta\in u$. We write $u\preceq_\uu v$.
		\item A $\uu$-closed set $u\subseteq \delta$ is \emph{$\uu$-straight} if every $\uu$-closed set $v\supseteq u$ is a $\uu$-straight extension of $u$. 
	\end{enumerate}
\end{definition}

Note that:
\begin{itemize}
	\item $\preceq_\uu$ is a transitive relation. Also, if $u\subseteq v \subseteq w$ are $\uu$-closed and $u\preceq_{\uu} w$ then $u\preceq_{\uu} v$. 
	\item A $\uu$-closed set $u\subseteq \delta$ is $\uu$-straight if and only if $u\preceq_\uu \delta$, i.e., if and only if for all $\beta\in S$, if $u\cap \beta$ is unbounded in $\beta$ then $\beta\in u$. 
	\item For all $\alpha\in S$, $\uu_\alpha\cup \{\alpha\}$ is $\uu$-straight. 
	\item If $u\subseteq \delta$ is $\uu$-closed and $\alpha\in \delta\setminus S$ then $u\cap \alpha \preceq_\uu u$. In particular, every $\alpha\in \delta\setminus S$ is $\uu$-straight. 
\end{itemize}

We also remark that in Lemma~\ref{lem:get_it_down_to_aleph_1} we can require $u$ to be $\uu$-straight, not merely $\uu$-closed. This is because every $\uu$-closed subset of $\delta$ of size $\aleph_1$ is contained in a $\uu$-straight subset of $\aleph_1$ of the same size; there are at most $\aleph_1$-many $\delta\in S$ which are limit points of~$u$; adding each of those, and for each such $\delta$, adding $\uu_\delta$, results in a set of size $\aleph_1$; repeating $\w$ times gives the desired~$\uu$-straight set.  
%

\begin{proposition}\label{prop:main_Cohen_propagation}
	Let $\bar \PP \in \PPP_\delta(S)$, witnessed by $\uu$. Let $u\subseteq v \subseteq \delta$ be $\uu$-closed. Suppose that:
	\begin{itemize}
		\item $u \preceq_\uu v$; and 
		\item $S$ does not reflect at any limit point of $v\setminus u$. 
		%
	\end{itemize}
	Then $\PP_v/\PP_u$ is equivalent to a Cohen algebra. 
\end{proposition}

Note that Proposition~\ref{prop:main_Cohen_propagation} implies~\eqref{item:Cohen} of Proposition~\ref{prop:main}.

\begin{proof}
	By induction on $\delta$. 
	
\smallskip
	
	First suppose that $\delta$ is a limit ordinal. Fix $u\subseteq v\subseteq \delta$ satisfying the hypotheses of the proposition.
	
	Suppose that $v\setminus u$ is bounded below $\delta$; let $\alpha = \sup (v\setminus u)$. By Corollary~\ref{cor:below_an_inaccessible_relent}, in $V^{\PP_u}$, $\PP_v/\PP_u$ is equivalent to $\PP_{v\cap \alpha}/\PP_{u\cap \alpha}$. The conditions of the proposition hold for the pair $(u\cap \alpha,v\cap \alpha)$ and so by induction, $\PP_{v\cap \alpha}/\PP_{u\cap \alpha}$ is equivalent to a Cohen algebra. 
	
	Suppose then that $v\setminus u$ is unbounded below $\delta$. Then~$S$ does not reflect at~$\delta$. Let~$c$ be a closed, unbounded subset of $\delta$ disjoint from $S$ (as mentioned above, this is by definition if $\cf(\delta)\ge \w_1$; otherwise, we use the fact that every element of $S$ has cofinality $\aleph_1$). 
Let $\alpha<\beta$ be elements of $c$. By Corollary~\ref{cor:2-dimensional_completeness}, 
	\[ \frac{\PP_{v\cap \beta}/\PP_{u\cap \beta}}{\PP_{v\cap \alpha}/\PP_{u\cap \alpha}}  = \PP_{v\cap \beta} / \PP_{(v\cap \alpha)\cup (u\cap \beta)}.\]
	Because $\alpha\notin S$, $(v\cap \alpha)\cup (u\cap \beta)\preceq_\uu (v\cap \beta)$. By induction, in $V^{\PP_{u\cap \beta}}$, 
		\[ \frac{\PP_{v\cap \beta}/\PP_{u\cap \beta}}{\PP_{v\cap \alpha}/\PP_{u\cap \alpha}} \]
		is equivalent to a Cohen algebra. Further, if $\beta$ is a limit point of $c$, then by Lemma~\ref{lem:limit_of_quotients},
		\[ \PP_{v\cap \beta}/\PP_{u\cap \beta} = \bigcup_{\alpha\in c\cap \beta} \PP_{v\cap \alpha} / \PP_{u\cap \alpha}. \]

		Let $G\subseteq \PP_u$ be generic; for $\alpha<\delta$, let $G_\alpha = G\cap \PP_{u\cap \alpha}$. Using Lemma~\ref{lem:Cohen_extensions}, by induction on $\alpha\in c\cup\{\delta\}$ we define an increasing and $\subseteq$-continuous sequence of sets $\seq{X_\alpha}$ and and increasing and continuous sequence of dense embeddings $\theta_\alpha\colon \PP_{v\cap \alpha}/G_\alpha\to \Cohen (X_\alpha)$.
		
		\

		Now suppose that the lemma is known for $\delta$. Fix $u\subseteq v\subseteq \delta+1$ satisfying the hypotheses of the proposition. For brevity, let $\bar u = u\cap \delta$ and $\bar v  = v\cap \delta$. By induction, $\PP_{\bar v}/\PP_{\bar u}$ is equivalent to a Cohen algebra. 

If $\delta\notin v$ then $v  = \bar v$ and $u = \bar u$. If $\delta\in u$ then by Corollary~\ref{cor:below_an_inaccessible_relent}, $\PP_{v}/\PP_u \sim \PP_{\bar v}/\PP_{\bar u}$. We suppose, then, that $\delta \in v\setminus u$, so $\bar u = u$.

Now there are two cases. If $\delta\notin S$ then $\PP_v = \PP_{\bar v}\times \Cohen$, so $\PP_v/\PP_u = (\PP_{\bar v}/\PP_u)\times \Cohen$ and so is equivalent to a Cohen algebra.

Suppose that $\delta\in S$. Then $u\preceq_\uu v$ implies that $u$ is bounded below $\delta$. Find some $\gamma\in [\sup u,\delta) \setminus S$ (recall that $S$ does not, for example, contain successor ordinals). Since $v$ is $\uu$-closed, $\uu_\delta\subset v$. We analyse $\PP_v/\PP_u$ in three steps. Let $y = u\cup \uu_\delta\cup \{\delta\}$. 
\begin{enumerate}
	\item $\PP_{y\cap \gamma}/\PP_u$ is equivalent to a Cohen algebra: this follows from induction and the fact that $y\cap \gamma\subseteq v$.
	\item $\PP_{y}/\PP_{y\cap \gamma}$ is equivalent to a Cohen algebra: by Lemma~\ref{lem:tensor_2}, $\PP_y/\PP_{y\cap \gamma}\sim \PP_{\uu_\delta\cup \{\delta\}}/\PP_{\uu_\delta\cap \gamma}$, because $\uu_\delta\cap \gamma = (y\cap \gamma)\cap (\uu_\delta\cup \{\delta\})$ (and the definition of $y$); we then use the assumption that $\bar \PP$ is $\uu$-quasi-Cohen.
	\item $\PP_v/\PP_y$ is equivalent to a Cohen algebra: here we note that since $u\preceq_\uu v$ and $\uu_\delta$ is $\uu$-straight, we have $y\preceq_\uu v$. Certainly $v\setminus y$ is bounded below ordinals at which $S$ reflects as $u\subseteq y$. Since $\delta\in y,v$ we are back in a previous case. \qedhere
\end{enumerate}
\end{proof}

\begin{remark}
	One way to think of Proposition~\ref{prop:main_Cohen_propagation} is by thinking of ``rearrangements'' or ``relistings'' of the coordinates. In the situation described, $v\setminus u$ can be re-ordered so that every initial segment (together with $u$) is $\uu$-closed, and the corresponding partial ordering is equivalent to a Cohen algebra. We do not pursue this formally here. 
\end{remark}

\section{Toward Cohen measurability} \label{sec:toward_measurability}

The following is known; we add a proof for completeness. 

\begin{proposition}\label{prop:jP_is_what_we_need}
	Let $j\colon V\to M = \Ult(V,\mu)$ be a normal ultrafilter embedding witnessing that $\kappa$ is a measurable cardinal. Let $\PP\subset V_\kappa$ be a notion of forcing, and suppose that $\PP\compembed j(\PP)$. Then in $V^\PP$ there is a $\kappa$-complete ideal $I$ on $\kappa$ such that $\Powerset(\kappa)/I\sim j(\PP)/\PP$. 
\end{proposition}

\begin{proof}
	Let $G\subset \PP$ be generic over $V$. In $V[G]$, let $\BB$ be a complete Boolean algebra such that $j(\PP)/G$ is a dense subset of $\BB$ (technically, its separative quotient is, we ignore this point).
	
	For any sentence $\vphi$ of the forcing language for $j(\PP)$ (with names from $M^{j(\PP)}$), let 
	\[ Y(\vphi) = \left\{ p\in j(\PP)/G\,:\, \left(p\force_{j(\PP)} \vphi\right)^M \right\} ,\]
and let $y(\vphi) = \sum^{\BB} Y(A)$. 

	Fact~\ref{fact:generic_compatibility_in_quotient} tells us that for any $\vphi$, $y(\lnot \vphi) = \lnot^{\BB} y(\vphi)$. Further, we note that if $g\in G$ and $g\in Y(\vphi)$ then $y(\vphi) = 1_\BB$, as $g$ is compatible with every $p\in j(\PP)/G$. 
	
	For a $\PP$-name $A$ for a subset of $\kappa$ we consider $y(\kappa\in j(A))$. Let $A$ and $B$ be such $\PP$-names, and suppose that $g\in G$ and $g\force_\PP A\subseteq B$. Since $j(g)=g$, we see that in $M$, $g\force_{j(\PP)} j(A)\subseteq j(B)$, so $y(j(A)\subseteq j(B))=1_{\BB}$. It follows that $y(\kappa\in j(A))\le y(\kappa\in j(B))$. This shows that the map $A\mapsto y(\kappa\in j(A))$ induces a function $\psi$ from $\Powerset(\kappa)^{V[G]}$ to $\BB$. 
	
	The observation above shows that for all $A\in \Powerset(\kappa)^{V[G]}$, $\psi(\kappa\setminus A) = \lnot^{\BB} \psi(A)$. Certainly $\psi(\emptyset) = 0_\BB$ and $\psi(\kappa) = 1_\BB$. 
	
	Let $\gamma<\kappa$, and let $\seq{D_i}_{i<\kappa}$ be a sequence of subsets of $\kappa$ in $V[G]$. We can fix a sequence $\seq{A_i}$ of $\PP$-names for $D_i$; we also fix a $\PP$-names $A_{\join}$ 
for $\bigcup_{i<\gamma} D_i$. 
In $V^{j(\PP)}$, $j(A_\join) = \bigcup_{i<\gamma} j(A_i)$. 
	
	Since for each $k<\gamma$, 
$A_k \subseteq A_\join$ in $V^\PP$, we know that 
$\psi(D_k) \le \psi(\bigcup_{i<\gamma} D_i)$, and so 
$\sum^\BB \psi (D_i) \le \psi(\bigcup_i D_i)$. To get equality, note that by Fact~\ref{fact:generic_compatibility_in_quotient}, $\bigcup_{i<\gamma} Y(\kappa\in j(A_i))$ is dense below each $p\in Y(\kappa\in j(A_\join))$ (in $j(\PP)/G$).

This shows that $\psi$ is a $\kappa$-complete Boolean homomorphism. So $I = \ker (\psi)$ is a $\kappa$-complete ideal on $\kappa$ in $V[G]$, and $\psi$ induces an embedding of $\Powerset(\kappa)/I$ into $\BB$ in $V[G]$. It remains to show that $\psi$ is dense. 

Let $p\in j(\PP)/G$. Recall that $\mu$ is the normal ultrafilter generating $j$; so $p = [\bar p]_\mu$ for some $\bar p \colon \kappa\to \PP$; write $\bar p = \seq{p_\alpha}_{\alpha<\kappa}$. Define a $\PP$-name $A$ for a subset of $\kappa$ by letting $A = \left\{ (p_\alpha,\alpha)  \,:\, \alpha<\kappa \right\}$. Then for all $\alpha<\kappa$ and $q\in \PP$, $q\force_\PP \alpha\in A$ if and only if $q\force_\PP p_\alpha\in G$. Hence in $M$, since $\kappa = [\id]_\mu$, for all $q\in j(\PP)/G$,  $q\force_{j(\PP)} \kappa\in j(A)$ if and only if $q\force_{j(\PP)} p\in G_{j(\PP)}$ if and only if in $\BB$, $q\le p$. Hence $\psi(A[G]) = p$.
\end{proof}

For the following proposition, let $\bar \PP\in \PPP_\kappa(S)$ and let $j\colon V\to M$ be an elementary embedding with critical point $\kappa$. Write $\PP$ for $\PP_\kappa$. Since every $p\in \PP$ has finite support (in particular, support bounded below $\kappa$), we see that essentially  $j\rest{\PP} = \id_{\PP}$ and that in $M$, $\PP = j(\PP)_\kappa$. Hence $\PP\compembed j(\PP)$ (in $M$, but this is absolute).

\begin{proposition}\label{prop:supercompact}
	Suppose that $\kappa$ is $2^\kappa$-supercompact. Suppose that $S$ only reflects at inaccessible cardinals. Let $\bar \PP\in \PPP_\kappa(S)$. Then there is a elementary embedding $j\colon V\to M$ given by a normal ultrafilter on $\kappa$ such that $j(\PP)/\PP$ is equivalent to a Cohen algebra. Hence, in $V^{\PP}$, $\kappa$ is real-valued Cohen. 
\end{proposition}

Note that the dense embedding of $j(\PP)/\PP$ into a Cohen algebra is not necessarily in $M^\PP$. 

\begin{proof}
Let $i\colon V\to N$ be an elementary embedding with critical point $\kappa$ such that $i(\kappa)> 2^\kappa$ and $N^{2^\kappa}\subset  N$. Let $U$ be the ultrafilter on $\kappa$ generated by $i$, i.e.\ $A\in U$ iff $\kappa\in i(A)$. Let $M = \Ult(V,U)$ be the transitive collapse of $V^\kappa/U$ and let $j\colon V\to M$ be the associated elementary embedding. The triangle can be completed by an elementary embedding $k\colon M\to N$, defined by $k([f]_U) = (i(f))(\kappa)$. So $i = k\circ j$. 
	
	Since $\kappa^\kappa\subset N$ and $U\in N$, we see that $k\rest{j(\kappa)} \in N$, and so $v = k[j(\kappa)]\in N$. In~$N$, $|v| = 2^\kappa$, so in $N$, $v\setminus \kappa$ is an Easton set (it is bounded below every inaccessible cardinal). The set $v$ is $i(\uu)$-closed, where $\uu$ witnesses that $\bar \PP\in \PPP_\kappa(S)$. For let $\gamma\in v$; let $\alpha = k^{-1}(\gamma)$. Then $k(j(\uu)_\alpha) = i(\uu)_\gamma$, and since $|j(\uu)_\alpha| = \aleph_1 < \kappa$, $k\rest {j(\uu)_\alpha}$ is a bijection between $j(\uu)_\alpha$ and $i(\uu)_\gamma$. It follows that $i(\uu)_\gamma\subset k[j(\kappa)] = v$. 	
	
	 Also note that $\kappa\notin i(S)$ ($\kappa$ is regular in $N$), and so $\kappa \preceq_{i(\uu)} v$. In $N$, $i(S)$ only reflects at inaccessible cardinals and $v\setminus \kappa$ is bounded below these. Hence, the conditions of Proposition~\ref{prop:main_Cohen_propagation} hold and we conclude that in $N$, $\PP\compembed i(\PP)_v$ and $i(\PP)_v/\PP$ is a Cohen algebra; this is upwards absolute, so holds in $V$. 
	
	We claim that $j(\PP)$ and $i(\PP)_v$ are isomorphic over $\PP$. Certainly, $k[j(\PP)]\subseteq i(\PP)$ and $k\rest{j(\PP)}$ is order-preserving. However, since $j(\PP)$ is a finite support iteration, for each $p\in j(\PP)$, $k$ pointwise maps the support of $p$ to the support of $k(p)$, and so $k(p)\in i(\PP)_v$. 
	
	It remains to show that $k$ is onto $i(\PP)_v$. By induction on $\delta\le j(\kappa)$, we show that $k\rest{j(\PP)_\delta}$ is onto $i(\PP)_{k[\delta]}$. This is preserved at limit stages since for limit $\delta\le i(\kappa)$, $i(\PP)_{v\cap \delta} = \bigcup_{\alpha<\delta} i(\PP)_{v\cap \alpha}$. Let $\delta<j(\kappa)$ and suppose that $k\rest{j(\PP)_\delta}$ is onto $i(\PP)_{k[\delta]} = i(\PP)_{v\cap k(\delta)}$. 
	
	As observed above, $k\rest {j(\uu)_\delta}$ is a bijection between $j(\uu)_\delta$ and $i(\uu)_{k(\delta)}$. It follows that $k\rest {j(\PP)_{j(\uu)_\delta}}$ is an isomorphism from $j(\PP)_{j(\uu)_\delta}$ to $i(\PP)_{i(\uu)_{k(\delta)}}$
	
	Recall that for some $R_\delta\subseteq H_{\w_1}$, $j(\PP)_{\delta+1} = j(\PP)_\delta \starr \QQ_\delta$, where $\QQ_\delta\subseteq j(\PP)_{j(\uu)_\delta}\times R_\delta$ is a $j(\PP)_{j(\uu)_\delta}$-name for an initial segment of $R_\delta$, and $j(\PP)_{\delta+1} = j(\PP)_\delta\starr \QQ_\delta$. Then $k(\QQ_\delta) = k[\QQ_\delta]$ is an $i(\PP)_{i(\uu)_{k(\delta)}}$-name for an initial segment of $k(R_\delta) = R_\delta$, and $i(\PP)_{v\cap k(\delta+1)} = i(\PP)_{v\cap k(\delta)}\starr k(\QQ_\delta)$. For $p\in j(\PP)_\delta$ and $s\in R$, $p\force_{j(\PP)_\delta} s\in \QQ_\delta$ if and only if $k(p)\force_{i(\PP)_{i(\uu)_{k(\delta)}}} s\in k(\QQ_\delta)$ and this shows that $k\rest{j(\PP)_{\delta+1}}$ is an isomorphism between $j(\PP)_{\delta+1}$ and $i(\PP)_{k[\delta+1]}$. 
\end{proof}

\section{Continuity on a non-meagre set} \label{sec:continuity}

\begin{definition}\label{def:funnyDiamond}
	Let~$\kappa$ be a cardinal and let $S\subseteq \kappa$ be stationary. We say that~$\funnyDiamond(S)$ holds if there is a sequence $\seq{f_\delta\,:\, \delta\in S}$ of functions such that for each $\delta\in S$, $\dom f_\delta\subseteq \delta$ is unbounded in $\delta$, $\range f_\delta\subset V_\delta$, and for every set $T\subseteq \kappa$ unbounded in~$\kappa$ and every function $F\colon T\to V_\kappa$, for stationarily many $\delta\in S$ we have \linebreak $F\rest {\dom f_\delta} =f_\delta$. 
\end{definition}

	The aim of this section is to prove~\eqref{item:continuity} of Proposition~\ref{prop:main}: if $\kappa$ is inaccessible (in fact $\kappa\ge \aleph_3,2^{\aleph_1}$ is sufficient), $S\subseteq S^\kappa_{\aleph_1}$ and $\funnyDiamond(S)$ holds, then there is $\bar \PP\in \PPP_\kappa(S)$ such that in~$V^{\PP_\kappa}$, every function from $2^\w$ to $2^\w$ is continuous on a non-meagre set. 

\

	We fix such $\kappa$ and $S$. Below, for brevity, for $\delta<\kappa$ we write $\PPP_\delta(S)$ for $\PPP_\delta(S\cap \delta)$. The construction of $\bar \PP$ (together with a witness template $\uu$) will be by induction, so we explain how to obtain pairs in $\PPP_\delta(S)$ from $\PPP_{\alpha}(S)$ for $\alpha<\delta$. The following are immediate from the definition of $\PPP_\delta(S)$:
	\begin{itemize}
		\item Let $\delta\le \kappa$ be a limit ordinal. Then $\PPP_\delta(S)$ is the set of all finite-support iterations $\bar \PP$ of length $\delta$ such that for some $S\cap \delta$-memory template $\uu$ of length $\delta$, for all $\alpha<\delta$, $\uu\rest \alpha$ witnesses that $\bar \PP\rest \alpha\in \PPP_\alpha(S)$. 
		\item Let $\delta<\kappa$ and let $\bar \PP\in \PPP_\delta(S)$, witnesses by $\uu$. Then $\uu\conc\seq{\emptyset}$ witnesses that $\bar \PP\times \Cohen\in \PPP_{\delta+1}(S)$. 
	\end{itemize}

	\begin{lemma}\label{lem:not_so_easy_successor}
		Let $\delta\in S$, and let $\bar \PP\in \PPP_\delta(S)$, witnessed by $\uu$. Suppose that:
		\begin{enumerate}
			\item $w\subseteq \delta$ is $\uu$-straight, and $|w|\le \aleph_1$;
			\item for some set $X$, $j\colon \PP_w\to \Cohen(\w_1,X)$ is a dense embedding;
			\item $c$ is an unbounded subset of $\delta$ which is disjoint from $S$;
			\item For all $\gamma\in c$, $j\rest{\PP_{w\cap \gamma}}$ is a dense embedding of $\PP_{w\cap \gamma}$ into $\Cohen(\epsilon_\gamma,X_\gamma)$ for some even $\epsilon_\gamma<\w_1$ and some set $X_\gamma\subset X$.
		\end{enumerate}
		For $i<\w_1$, let $\eta_i$ be the $\PP_u$-name for a real which is the $j$-pullback of the Cohen name given by $\Cohen(\{2i\})$, and similarly let $\zeta_i$ be the pullback of the name given by $\Cohen(\{2i+1\})$. Then $\uu\conc\{w\}$ witnesses that $\bar \PP \starr \Shelah(\bar\eta,\bar \zeta)\in \PPP_{\delta+1}(S)$. 
	\end{lemma}

	\begin{proof}
		We need to show that for all $\beta\in \delta\setminus S$, $\PP_{w}\starr \Shelah(\bar\eta,\bar \zeta) / \PP_{w\cap \beta}$ is equivalent to a Cohen algebra. Fix such $\beta$. Find some $\gamma\in c$ greater than $\beta$. Since $\beta\notin S$, $w\cap \beta \preceq_{\uu} w\cap \gamma$. Since $|w|\le \aleph_1$ and $S$ does not reflect at any ordinal of cofinality~$\aleph_1$ (as $S\subseteq S^\kappa_{\aleph_1}$), $w$ is bounded below any ordinal at which $S$ reflects. Hence, by Proposition~\ref{prop:main_Cohen_propagation}, $\PP_{w\cap \gamma} / \PP_{w\cap \beta}$ is equivalent to a Cohen algebra. 
		
		$\PP_{w}\starr \Shelah(\bar\eta,\bar \zeta) / \PP_{w\cap \gamma}$ is isomorphic to $\Cohen (X\setminus X_\gamma)\times \left( \Cohen(\w_1)\starr \Shelah \right)/\Cohen(\epsilon_\gamma)$, and so by Proposition~\ref{prop:Shelah_properties} is equivalent to a Cohen algebra. 
\end{proof}

	\begin{lemma}\label{lem:one-step_meagreness}
		Let $\bar \PP\in \PPP_\kappa(S)$, witnessed by $\uu$. Suppose that $p\in \PP_\kappa$ forces that $A\subset 2^\w$ has size~$\aleph_1$ and is meagre. Then there is some $\uu$-closed $w\subset \kappa$ of size $\aleph_1$ such that $p\in \PP_w$ and $A$ is a $\PP_w$-name and such that $p\force_{\PP_w} ``\text{$A$ is meagre}"$.
	\end{lemma}	

	\begin{proof}
		This is similar to the proof of Lemma~\ref{lem:get_it_down_to_aleph_1}. There is a countable sequence $\seq{T_n}$ of~$\PP_\kappa$-names for trees in Cantor space (subtrees of $2^{<\w}$) such that~$p$ forces that no~$T_n$ contains a clopen set, and $A\subseteq \bigcup_n [T_n]$. We pass to an elementary submodel~$M$ of the universe of size~$\aleph_1$ containing all pertinent objects. Let $\seq{\eta_\alpha\,:\, \alpha<\w_1}$ be a sequence of~$\PP_\kappa$-names for reals such that $p$ forces that $A = \{\eta_\alpha\,:\, \alpha<\w_1\}$. For each~$\alpha$ and each~$n$, the set of conditions $q\in \PP_\kappa$ for which there is some finite binary string~$\s$ of length~$n$ such that~$q$ forces in~$\PP_\kappa$ that~$\s\in T$ and $\s\prec \eta_\alpha$ is dense below~$p$, and so dense below~$p$ in $\PP_w = \PP_\kappa\cap M$ where $w = \kappa\cap M$. Since $\PP_w\compembed \PP_\kappa$ and these statements are absolute, forcing them in~$\PP_\kappa$ and in~$\PP_w$ are equivalent. 		 
	\end{proof}

	\begin{lemma}\label{lem:meagreness_preserved}
		Let $\bar \PP\in \PPP_\kappa(S)$, witnessed by the $S$-memory template $\uu$. Let $v\subset \kappa$ be $\uu$-straight and suppose that in $V^{\PP_v}$, $A\subseteq 2^\w$ is non-meagre and has size $\aleph_1$. Then $A$ is also non-meagre in $V^{\PP_\kappa}$. 
	\end{lemma}

	\begin{proof}
		If not, then there is some $p\in \PP_\kappa$ which forces that $A$ is meagre. By Lemma~\ref{lem:one-step_meagreness}, there is some $\uu$-closed $w\supseteq v$ such that $|w\setminus v| \le \aleph_1$, $p\in \PP_w$ and $p\force_{\PP_w}``\text{$A$ is meagre}"$. This is impossible, since $\PP_w/\PP_v$ is a Cohen extension.
	\end{proof}

For $\delta\le \kappa$, $\bar \PP\in \PPP_\delta(S)$ (witnessed by $\uu$) and $\alpha\in \delta\setminus S$, we let $\rho_\alpha$ be the name for the Cohen real added by $\QQ_\alpha = \PP_{\{\alpha\}}$. Since $\rho_\alpha$ is a $\PP_{\{\alpha\}}$-name, it is a $\PP_v$-name for any $\uu$-closed set $v$ containing $\alpha$. 

\begin{lemma}\label{lem:main_step}
	Let $\delta\in S$, and let $\bar \PP\in \PPP_\delta(S)$, witnessed by $\uu$.
	
	Suppose that $\{\alpha_i\,:\,i<\w_1\}$ is an increasing enumeration of an unbounded subset of $\delta$, disjoint from $S$, and that we have:
		\begin{itemize}
			\item for $i<\w_1$, $\uu$-straight sets $w_i$ of size $\aleph_1$ such that $\alpha_i\in w_i$, $w_i\subset \alpha_{i+1}$ and $w_i\cap \alpha_i$ is a constant $w^*$;
			\item sets $Y,Z,W$ and dense embeddings $j_i$ of $\PP_{w_i}$ into $\Cohen(Y,Z,W)$, such that $j_i\rest {\PP_{w^*}}$ is constant, a dense embedding into $\Cohen(Y)$; and $j_i\rest {\QQ_{\alpha_i}}$ is 
		a dense embedding into $\Cohen(Z)$ such that $j_i[\rho_{\alpha_i}]$ is the Cohen real added by $\Cohen(Z)$;
			\item A $\PP_{w_i}$-name $\eta_i$ for a real such that the push-forward $j_i[\eta_i]$ (which is a $\Cohen(Y,Z,W)$-name for a real) is a constant~$\nu^*$.
			\end{itemize}
		Let $w = \bigcup_i w_i$. Then $w$ is $\uu$-straight and there is a $\PP_w$-name $\QQ$ for a notion of forcing such that:
	\begin{itemize}
		\item $\bar \PP\starr \QQ\in \PPP_\delta(S)$ (witnessed by $\uu\conc\seq{w}$); and
		\item In $V^{\PP_w\starr \QQ}$ there is a non-meagre subset $A$ of $\{\rho_{\alpha_i}\,:\, i<\w_1\}$ such that the map $\rho_{\alpha_i}\mapsto \eta_i$ is continuous on $A$. 
	\end{itemize}
\end{lemma}

\begin{proof}
	Since $\Cohen(Y,Z,W)$ is c.c.c., there is a countable set $X\subseteq W$ such that $\nu^*$ is a $\Cohen(Y,Z,X)$-name. We replace $W$ by $W\setminus X$. Also note that $Z$ is countable. 
	
	We merge the embeddings $j_i$ as follows. We assume that $Y$ and $W$ are disjoint from $\w_1$. For $i<\w_1$ fix isomorphisms $f_i$ of $\Cohen(Z)$ with $\Cohen(\{2i\})$ and $g_i$ of $\Cohen(X)$ with $\Cohen(\{2i+1\})$. Also fix sets $A_i$ (pairwise disjoint, and disjoint from $Y$ and $\w_1$) and isomorphisms $h_i$ of $\Cohen(W)$ with $\Cohen(A_i)$. Let $A_{<i} = \bigcup_{i'<i} A_{i'}$ and $A = A_{<\w_1}$.
	
	 The product map $\vphi_i = (\id_{\Cohen(Y)}, f_i,g_i,h_i)$ is an isomorphism of $\Cohen(Y,Z,X,W)$ with $\Cohen(Y,\{2i,2i+1\},A_i)$. The composition $\psi_i = \vphi_i\circ j_i$ is a dense embedding of $\PP_{w_i}$ into $\Cohen(Y,\{2i,2i+1\},A_i)$. The map $\psi_i\rest {\PP_{w^*}}$ equals $j_i\rest{\PP_{w^*}}$ (a constant, dense embedding into $\Cohen(Y)$). Letting $\nu_i$ be the name for the Cohen real added by $\Cohen(\{i\})$, we have $\psi_i[\rho_{\alpha_i}] = \nu_{2i}$. 
	
	$w$ is $\uu$-straight since every limit point of $w$ is a limit point of some $w_i$. Let $p\in \PP_w$. For all but finitely many $i<\w_1$ we have $p\rest{w_i}\in \PP_{w^*}$. In other words,~$\PP_{w}$ is the finite support product of the~$\PP_{w_i}$ over the root~$\PP_{w^*}$. Thus, the sequence \linebreak $\psi(p) = \seq{\psi_i(p)\,:\, i<\w_1}$ is an element of the finite support product of \linebreak $\Cohen(Y,\{2i,2i+1\},A_i)$ over $\Cohen(Y)$ which equals $\Cohen(Y,\w_1,A)$, and $\psi$ is a dense embedding of $\PP_w$ into $\Cohen(Y,\w_1,A)$. Further, for $i<\w_1$, $\psi\rest{\PP_{w\cap \alpha_i}}$ is a dense embedding of $\PP_{w\cap \alpha_i} = \PP_{\bigcup_{i'<i} w_{i'}}$ into $\Cohen(Y,2i,A_{<i})$. 

	The isomorphism $\vphi_i$ carries the name $\nu^*$ to a $\Cohen(Y,\{2i,2i+1\})$-name for a real $\vphi_i[\nu^*]$. In $V^{\Cohen(Y)}$ there is a continuous function which in $V^{\Cohen(Y,\w_1)}$ maps the pair $(\nu_{2i},\nu_{2i+1})$ to $\vphi_i[\nu^*] = \psi_i(\eta_i)$. Translating back, letting $\mu_i = \psi^{-1}[\nu_{2i+1}]$, in $V^{\PP_{w^*}}$ there is a continuous function which in $V^{\PP_w}$ takes the pair $(\rho_{\alpha_i},\mu_i)$ to $\eta_i$. 
	
	Let $\QQ = \Shelah(\seq{\rho_{\alpha_i}},\seq{\mu_i})$. Then in $V^{\PP_w\starr \QQ}$ there is a non-meagre subset $A$ of \linebreak $\{\rho_{\alpha_i}\,:\, i<\w_1\}$ on which the map $\rho_{\alpha_i}\mapsto \mu_{i}$ is continuous. On $A$, the map $\rho_{\alpha_i}\mapsto \eta_i$ is the composition of two continuous functions and so is continuous. Finally, Lemma~\ref{lem:not_so_easy_successor} shows that $\bar \PP\rest \delta \starr \QQ\in \PPP_{\delta+1}(S)$.
\end{proof}

We now prove~\eqref{item:continuity} of Proposition~\ref{prop:main}. Let $\seq{f_\delta\,:\, \delta\in S}$ witness $\funnyDiamond(S)$. 

By recursion we define $\bar \PP\in \PPP_\kappa(S)$ with a witness template $\uu$. At step $\delta<\kappa$, say we have already defined $\uu\rest \delta$ and $\bar \PP\rest \delta$ (taking limits at limit stages). In the interesting case, suppose that $\delta\in S$. If $\dom f_\delta$ is a set increasingly enumerated as $\{\alpha_i\,:\, i<\w_1\}$ and $f_\delta(\alpha_i) = (w_i,Y,Z,W,j_i,\eta_i)$ where the conditions of Lemma~\ref{lem:main_step} hold, then we choose $\uu_{\delta} = w$ and $\QQ_\delta = \QQ$ for $(w,\QQ)$ given by the lemma. Otherwise we let $\uu_\delta = \emptyset$ and $\QQ_\delta = \Cohen$. 

This defines $\uu$ and $\bar \PP\in \PPP_\kappa(S)$. Now let $F$ be a $\PP_\kappa$-name for a function from $2^\w$ to~$2^\w$. Again recall that for $\alpha\in \kappa\setminus S$, $\rho_\alpha$ is the $\QQ_\alpha = \PP_{\{\alpha\}}$-name for the Cohen real added. 

By Lemma~\ref{lem:get_it_down_to_aleph_1} (and the discussion after Definition~\ref{def:straight}), for $\alpha\in \kappa\setminus S$ let $v_\alpha$ be a $\uu$-straight set of size $\aleph_1$ such that $F(\rho_\alpha)$ is a $\PP_{v_\alpha}$-name. By increasing, we may assume that $\alpha\in v_\alpha$. 

Since $|v_\alpha|= \aleph_1$, it is bounded below any ordinal at which $S$ reflects. Then $\PP_{v_\alpha\cap \alpha}$ is equivalent to a Cohen algebra (of dimension at most $\aleph_1$). Of course $\PP_{v_\alpha\cap (\alpha+1)}/ \PP_{v_\alpha\cap \alpha} = \PP_{\{\alpha\}}$ is equivalent to a Cohen algebra (of dimension 1), and since $\alpha+1\notin S$, $\PP_{v_\alpha}/\PP_{v_\alpha\cap (\alpha+1)}$ is equivalent to a Cohen algebra. We can therefore find sets $Y_\alpha,Z_\alpha,W_\alpha\in V_{\w_1}$ (all of size at most $\aleph_1$, and $Z_\alpha$ countable (or a singleton)) and a dense embedding $k_\alpha$ of $\PP_{v_\alpha}$ into $\Cohen(Y_\alpha,Z_\alpha,W_\alpha)$ such that $k_\alpha\rest{\PP_{v_\alpha\cap \alpha}}$ is a dense embedding into $\Cohen(Y_\alpha)$ and $k_\alpha\rest{\PP_{\{\alpha\}}}$ is the canonical dense embedding into $\Cohen(Z_\alpha)$ (so $k_\alpha[\rho_\alpha]$ is the Cohen real added by $\Cohen(Z_\alpha)$). We let $\nu_\alpha = k_\alpha[F(\rho_\alpha)]$. 

Since $\kappa\setminus S$ is stationary, there is some unbounded (indeed stationary) set $T\subset \kappa$ such that for $\alpha\in T$, $Y_\alpha,Z_\alpha,W_\alpha$ and $\nu_\alpha$ are constant $Y,Z,W$ and $\nu^*$; and further, since $\kappa$ is inaccessible, $v_\alpha\cap \alpha$ is a constant $w^*$ and $j_\alpha\rest{\PP_{w^*}}$ is a constant $j^*$.
 By induction we choose an unbounded $T'\subset T$ such that for $\alpha<\beta$ from $T'$, $v_\alpha\subset \beta$. For $\alpha\in T'$, let $h(\alpha) = (v_\alpha,Y,Z,W,j_\alpha,F(\rho_\alpha))$. 

There is some $\delta\in S$ such that $f_\delta = h\rest{\dom f_\delta}$. By restricting to a subset we may assume that $\dom f_\delta  = \{\alpha_i\,:\, i<\w_1\}$ (increasing enumeration). 
Then letting $w_i = v_{\alpha_i}$ and $\eta_i = F(\rho_{\alpha_i})$, we see that the conditions of Lemma~\ref{lem:main_step} hold, and so, for the resulting $w$, in $V^{\PP_{w\cup\{\delta\}}}$ there is a non-meagre set $A\subseteq \{\rho_{\alpha_i}\,:\, i<\w_1\}$ on which $\rho_{\alpha_i}\mapsto \eta_i = F(\rho_{\alpha_i})$ is continuous. By Lemma~\ref{lem:meagreness_preserved}, $A$ is non-meagre in $V^{\PP_\kappa}$ as well. This concludes the proof.

\section{Consistency strength} \label{sec:consistency}
	We prove~\eqref{item:preparation} of Proposition~\ref{prop:main}. Several ideas come from~\cite{GitikShelah}. Let $\kappa$ be a measurable cardinal. We want to show that in a forcing extension $W$ there are $S\subseteq S^\kappa_{\aleph_1}$ and a normal ultrafilter embedding $j\colon W\to N$ with critical point $\kappa$ such that:
	\begin{enumerate}
		\item $\funnyDiamond(S)$ holds; and
		\item in $W$, $j(S)$ reflects no-where in the interval $(\kappa,j(\kappa)]$. 
	\end{enumerate}
	
We first add $S$ and its ``quasi-diamond'' sequence.

 \begin{notation}
	 Let $I$ be the class of inaccessible cardinals and let $\bar I$ be its closure. For $\alpha<\kappa$ let $\alpha^\iplus$ be the least element of $I$ greater than $\alpha$. 
 \end{notation}

\begin{definition} \label{def:Q}
	\textbf{1.} Let $\QQ$ consists of the pairs $p = (\s,{F})$ such that:
	\begin{itemize}
		\item $\s$ is a partial function from $\kappa$ to $2$. The domain of $\s$ is an Easton set (bounded below each inaccessible cardinal). Further, for all $\delta\in \bar I$, the domain of 
	$\s\rest{[\delta,\delta^\iplus)}$ is an initial segment of $[\delta,\delta^\iplus)$. 
		\item Every $\alpha\in \s^{-1}\{1\}$ has cofinality $\aleph_1$.
		\item For all $\alpha<\kappa$ of uncountable cofinality, if $\dom \s$ is unbounded in $\alpha$, then $\s^{-1}\{0\}$ contains a club of $\alpha$.
		\item $F = \seq{f_\alpha:\alpha\in \s^{-1}\{1\}}$ is a sequence of functions such that for all $\alpha\in \s^{-1}\{1\}$, $\dom f_\alpha$ is a club of $\alpha$ of order-type $\w_1$ and $\range f_\alpha\subset \alpha$.
	\end{itemize}
	Extension in $\QQ$ is given by extension in both coordinates. 

	\textbf{2.} As defined, $\QQ$ is a class of conditions, and we will use only a set of these. For an interval of ordinals $A$, we let $\QQ_A$ be the set of conditions $(\s,F)\in \QQ$ such that $\dom \s\subseteq A$. We will be mainly interested in the case $A = \kappa$. 
\end{definition}

\begin{definition}
	Let $p = (\s,F)\in \QQ_A$. We define a condition $\bar p = (\bar \s,F)$ by defining $\bar \s(\beta) =0$ for all $\beta\in A\setminus \dom \s$ such that $\beta\cap \dom \s$ is unbounded in $\beta$. Since we are only adding to $\dom \bar \s$ limit points of $\dom \s$, it follows that the limit points of $\dom \s$ and $\dom \bar \s$ are the same, and so $\bar p\in \QQ_A$.
\end{definition}
 
\medskip

Let $\gamma$ be an ordinal, and suppose that $\seq{p_i}_{i<\gamma}$ is a decreasing sequence from~$\QQ$. If $p_i = (\s_i,F_i)$ then we let $\s_{<\gamma} = \bigcup_{i<\gamma} \s_i$, $F_{<\gamma} = \bigcup_{i<\gamma} F_i$, and $p_{<\gamma} = (\s_{<\gamma},F_{<\gamma})$. If $p_i\in \QQ_A$ for all $i<\gamma$, and $\cf(\gamma) < (\min A)^\iplus$ it may still be the case that $p_{<\gamma}\notin \QQ$, because of the non-reflection requirement. However, if in addition the sequence witnesses non-reflection the we get closure. 

\begin{lemma}\label{lem:essential_strategic_closure}
	Let $A$ be an interval of ordinals. Let $\gamma < (\min A)^\iplus$, and let $\seq{p_i}_{i<\gamma}$ be a decreasing sequence of conditions in $\QQ_A$. Further suppose that for all $i<\gamma$,~$p_i$ extends~$\overline{p_{<i}}$. Then $p_{<\gamma}\in \QQ_A$.
\end{lemma}

\begin{proof}
	We only need to check the non-reflection condition. Let $\alpha$ have uncountable cofinality and suppose that $\dom \s_{<\gamma}$ is unbounded in $\alpha$. We assume that for all $i<\gamma$, $\alpha\cap \dom \s_{i}$ is bounded below $\alpha$; let $\beta_i = \sup (\dom \s_{<i}\cap \alpha)$. The sequence $\seq{\s_{<i}}$ is continuous so $\{\beta_i\,:\, i<\gamma\}$ is a club of $\alpha$. By restricting to a club of $i\in \gamma$, we may assume that $\beta_i\notin \dom \s_{<i}$, and in this case $\s_i(\beta_i) = \overline{\s_{<i}}(\beta_i) = 0$.  Thus, a club subset of the $\beta_i$ ensures that $\alpha$ does not prevent $p_{<\gamma}$ from being in $\QQ_A$. 
\end{proof}

This shows that $\QQ_A$ is $<(\min A)^\iplus$-strategically directed closed, and so $<(\min A)^\iplus$-distributive. If $\lambda<\kappa$ is inaccessible, then $|\QQ_\lambda|= \lambda$, and $\QQ_\kappa = \QQ_\lambda \times \QQ_{[\lambda,\kappa)}$. Since $\QQ_{[\lambda,\kappa)}$ is $\lambda^+$-distributive, this shows that $\QQ_\kappa$  preserves all cofinalities (and so, preserves all cardinals). Similarly, for any cardinal $\delta<\kappa$, since $|\QQ_{\delta}| < \delta^\iplus$, $(2^{\delta})^{V_{\QQ_\delta}}< \delta^\iplus$, and $\QQ_{[\delta,\kappa)}$ does not add subsets of $\delta$. We conclude that any inaccessible cardinal in $V$ is also inaccessible in $V^{\QQ_\kappa}$. 
%
%
%


\begin{proposition}\label{prop:measurability_preserved}
	Assume $2^\kappa = \kappa^+$. Then $\kappa$ is measurable in $V^{\QQ_\kappa}$. 
	%
	%
\end{proposition}

This is fairly standard, but we include a proof for completeness. 

\begin{proof}
	Let $j\colon V\to M$ be a normal ultrafilter elementary embedding with critical point $\kappa$; so $M^\kappa\subset M$. In $M$, $j(\QQ_\kappa) = \QQ_{j(\kappa)} = \QQ_\kappa\times \QQ_{[\kappa,j(\kappa))}$, where $\QQ_\kappa$ is absolute between $V$ and $M$, and $j$ is the identity on $\QQ_\kappa$. 
	
	Let $G\subset \QQ_\kappa$ be generic over $V$. Let $\seq{A_i}_{i<\kappa^+}$ be an enumeration of all $\QQ_\kappa$-names for subsets of $\kappa$. In $V[G]$ we construct a sequence $\seq{p_i}_{i<\kappa^+}$ of conditions in $\QQ^M_{[\kappa,j(\kappa))}$ (with each $p_i$ extending $\overline{p_{<i}}$), such that for some $g\in G$, $(g,p_i)$ decides (in $M$, for the forcing $\QQ^M_{j(\kappa)}$) whether $\kappa\in j(A_i)$ or not. If $k<\kappa^+$ then, as $M[G]^{\kappa}\subset M[G]$ in $V[G]$, the sequence $\seq{p_i}_{i<k}$ is in $M[G]$; and as in $M$, $\QQ^M_{[\kappa,j(\kappa))}$ is $<\kappa^+$-distributive, the sequence $\seq{p_i}_{i<k}$ is in fact in $M$, and so its limit $p_{<k}$ is also in $M$, and the construction can proceed. This defines, in $V[G]$, a $\kappa$-complete ultrafilter on $\kappa$. The point is that if $J\subset \kappa^+$ has size at most $\kappa$, and for all $i\in J$, $\left((g,p_i)\force \kappa\in j(A_i)\right)^M$ for some $g\in G$, then for $k<\kappa^+$ such that $\force_{\QQ_\kappa} A_k = \bigcap_{i\in I} A_i$, for no $g\in G$ can $(g,p_k)$ force that $\kappa\notin j(A_k)$. 
\end{proof}

If $G\subset \QQ_\kappa$ is a generic filter, we let $S = S[G]$ be the union of $\s^{-1}\{1\}$ where $(\s,F)\in G$ for some $F$.

\begin{lemma}\label{lem:diamond}
	In $V^{\QQ_\kappa}$, $S$ is stationary and $\funnyDiamond(S)$ holds. 
\end{lemma}

\begin{proof}
	The $\funnyDiamond(S)$ sequence is a modification of the sequence given generically by the second coordinate of the forcing conditions. Let $\seq{f_\delta}_{\delta\in S}$ be the union of of $F$ where $(\s,F)\in G$ for some $\s$. We show that in $V^{\QQ_\kappa}$, for any function $h$ from an unbounded subset of $\kappa$ to $\kappa$, there are stationarily many $\delta\in S$ such that $h\rest{\dom f_\delta} = f_\delta$. To then capture functions into $V_\kappa$ we compose the functions $f_\delta$ with some fixed bijection between $\kappa$ and $V_\kappa$. 
	
	Let~$C$ be a $\QQ_\kappa$-name for a club of $\kappa$, and let $h$ be a $\QQ_\kappa$-name for a function from an unbounded subset of $\kappa$ to $\kappa$. 
	
	 Starting from an arbitrary $p_0\in \QQ_\kappa$, we define a decreasing sequence $\seq{p_i}_{i<\w_1}$ with each $p_i$ extending $\overline{p_{<i}}$, and three increasing sequences $\seq{\alpha_i}_{i<\w_1}$, $\seq{\zeta_i}_{i<\w_1}$ and $\seq{\gamma_i}_{i<\w_1}$ such that:
	\begin{itemize}
		\item $\alpha_i < \gamma_i < \alpha_{i+1}$ and $\zeta_i < \alpha_{i+1}$;
		\item $p_i\force_{\QQ_\kappa} \alpha_i\in C$;
		\item $p_i\force_{\QQ_\kappa} \gamma_i\in \dom h \andd h(\gamma_i) = \zeta_i$. 
	\end{itemize}
	Let $\alpha^* = \sup_{i<\w_1} \alpha_i$. Define $p^*$ by extending $p_{<\w_1}$ by letting $\s^*(\alpha^*) = 1$ and $f_{\alpha^*} = \seq{\gamma_i\mapsto \zeta_i}$. Then~$p^*$ forces that~$S$ intersects~$C$ at $\alpha^*$ and that $h\rest{\dom f_{\alpha^*}}= f_{\alpha^*}$. 
		%
		%
\end{proof}

\begin{lemma}\label{lem:only_inaccessibles}
	In $V^{\QQ_\kappa}$, $S$ only reflects at inaccessible cardinals. 
\end{lemma}

\begin{proof}
	Let $\alpha<\kappa$ of uncountable cofinality be accessible in $V^{\QQ_\kappa}$; then it is accessible in $V$ as well. Let $p = (\s,F)\in \QQ_\kappa$. If $\dom \s\cap \alpha$ is unbounded in $\alpha$, then since $\s^{-1}\{0\}$ contains a club of $\alpha$, $p$ forces that $S$ does not reflect at $\alpha$. The collection of condition $p = (\s,F)$ such that $\dom \s\cap \alpha$ is unbounded in $\alpha$ is dense in $\QQ_\kappa$; we can always extend any given condition by sufficiently many zeros (on an interval if $\alpha$ is not a limit of inaccessible cardinals, and on a club of $\alpha$ of order-type $\cf(\alpha)$ otherwise). 
\end{proof}

\

Our next step is to extend $V^{\QQ_\kappa}$ to a model in which $j(S)$ reflects no-where in $(\kappa,j(\kappa)]$. As this is mirrored below $\kappa$, we look at $\lambda<\kappa$ first. 

\medskip

We work in $V$. Let $\lambda<\kappa$. We let $R_{\lambda}$ be the collection of all functions $C$ whose domain is $I\cap (\lambda,\kappa]$ such that for each $\delta\in \dom C$, $C(\delta)\subseteq (\lambda,\delta)$ is a closed set, bounded below $\delta$ (so $\sup C(\delta)\in C(\delta)$). We order $R_\lambda$ coordinatewise by end extension. 
%
%
We define a $\QQ_{[\lambda,\kappa)}$-name $\SSS_\lambda$ for an upward-closed subset of $R_\lambda$. This name consists of the pairs $((\s,F),C)\in \QQ_{[\lambda,\kappa)}\times R_\lambda$ such that for all $\delta\in I\cap (\lambda,\kappa]$, $C(\delta) \subseteq \s^{-1}\{0\}$. Note that since $\QQ_\kappa = \QQ_\lambda\times \QQ_{[\lambda,\kappa)}$, $\SSS_\lambda$ is also a $\QQ_\kappa$-name. 

%
%

\begin{proposition}\label{prop:S1}
	In $V^{\QQ_\kappa\starr \SSS_\lambda}$, $S$ reflects no-where in the interval $(\lambda,\kappa]$. 
\end{proposition}

\begin{proof}
	If $\delta\in(\lambda,\kappa]$ is an accessible ordinal of uncountable cofinality, then by Lemma~\ref{lem:only_inaccessibles}, in $V^{\QQ_\kappa}$, $S$ is disjoint from a club of $\delta$ in $V^{\QQ_\kappa}$ and hence also in any extension of $V^{\QQ_\kappa}$.
	
	Let $\delta\in (\lambda,\kappa]$ be inaccessible. As $V^{\QQ_{[\lambda,\kappa)}\starr \SSS_\lambda}\subset V^{\QQ_\kappa\starr \SSS_\lambda}$, it is sufficient to show that~$S_{[\lambda,\kappa)}$ (the generic subset of $[\lambda,\kappa)$ added by $\QQ_{[\lambda,\kappa)}$) does not reflect at~$\delta$ in $V^{\QQ_{[\lambda,\kappa)}\starr \SSS_\lambda}$. If~$G$ is a generic filter for $\QQ_{[\lambda,\kappa)}\starr \SSS_\lambda$ over~$V$ then we let~$D_\delta$ be the union of $C(\delta)$ where $(p,C)\in G$ for some~$p$. It is clear that~$D_\delta$ is disjoint from~$S_{[\lambda,\kappa)}$, and that~$D_\delta$ is a closed subset of $\sup D_\delta$. We need to show that $\sup D_\delta = \delta$. 
	
	We work in $V$. Let $\gamma<\delta$; let $((\s,F),C)$ be any condition in $\QQ_{[\lambda,\kappa)}\starr \SSS_\lambda$. Since $\delta \cap \dom \s$ is bounded below $\delta$, it is easy to extend $\s$ to $\s'$ by adding sufficiently many zeros so that $\gamma\in \dom \s'$ and extend $C$ to $C'$ so that $C'(\delta)$ intersects the interval $[\gamma,\delta)$. 	
\end{proof}

\begin{proposition}\label{prop:S2}
	In $V^{\QQ_\kappa}$, $\SSS_\lambda$ is $<\lambda^\iplus$-distributive.
\end{proposition}

\begin{proof}
	It suffices to show that $\QQ_{[\lambda,\kappa)}\starr \SSS_\lambda$ is $<\lambda^\iplus$ distributive (in $V$). For then, since $|\QQ_\lambda|= \lambda$, it has the same property in $V^{\QQ_\lambda}$, which in turn implies that for $\gamma<\lambda^\iplus$, all sequences of ordinals of length $\gamma$ in $V^{\QQ_\kappa\starr \SSS_\lambda}$ in fact lie in $V^{\QQ_\lambda}$ and so in~$V^{\QQ_\kappa}$. 
	
		Work in $V$. Given a condition $q = (p,C)$ in $\QQ_{[\lambda,\kappa)}\starr \SSS_\lambda$, we define $\bar q = (\bar p,\bar C)$ by taking $\bar p$ as above, and defining, for each limit point $\beta$ of $\dom \s$ (where $p = (\s,F)$) which is not already in $\dom \s$, for each inaccessible $\delta > \beta$ such that $\dom \s\cap [\beta,\delta) = \emptyset$, $\bar C(\delta) = C(\delta)\cup\{\beta\}$. Since for such $\delta$ we have $C(\delta)\subseteq \beta$, $\bar C(\delta)$ is indeed an end-extension of $C(\delta)$, and it is clear that $\bar q\in \QQ_{[\lambda,\kappa)}\starr \SSS_\lambda$.
		
		We show that a dense subset of $\QQ_{[\lambda,\kappa)}\starr \SSS_\lambda$ is $<\lambda^\iplus$-strategically closed. Let $T$ be the set of conditions $((\s,F),C)\in \QQ_{[\lambda,\kappa)}\starr \SSS_\lambda$ such that:
		\begin{itemize}
			\item $\dom \s$ is a closed subset of $\kappa$; and
			\item For all $\delta\in I\cap (\lambda,\kappa]$, $\max C(\delta) = \sup (\dom \s\cap \delta)$. 
		\end{itemize}
		Extending a given condition $((\s,F),C)$ from $\QQ_{[\lambda,\kappa)}\starr \SSS_\lambda$ to a condition $((\s',F),C')$ in $T$ is not difficult; for each $\delta\in I\cap (\lambda,\kappa]$ we add $\epsilon =\sup (\dom \s\cap \delta)$ and also $\epsilon+1$ to $\dom \s'$, and let $\s'(\epsilon+1)=0$ and $\epsilon+1 \in C'(\delta)$. 
		
		Let $\gamma<\lambda^\iplus$ and that $\seq{q_i}_{i<\gamma}$ is a decreasing sequence of conditions from $T$ such that for all $i<\gamma$, $q_i$ extends $\overline{q_{<i}}$. Then $\overline{q_{<\gamma}}\in T$. 
\end{proof}

\medskip

We can now prove~\eqref{item:preparation} of Proposition~\ref{prop:main}. Starting with the measurable cardinal~$\kappa$ such that $2^\kappa = \kappa^+$, we first work in $V^{\QQ_\kappa}$. Since $\kappa$ is measurable in $V^{\QQ_\kappa}$ (Proposition~\ref{prop:measurability_preserved}), let $\mu$ be a normal ultrafilter on $\kappa$ in $V^{\QQ_\kappa}$, and let $j$ be the associated embedding from $V^{\QQ_\kappa}$ into the transitive inner model $P = \left( \Ult(V^{\QQ_\kappa},\mu) \right)$ of~$V^{\QQ_\kappa}$. Using $\mu$ to average the notions of forcing $\SSS_\lambda$ for inaccessible $\lambda<\kappa$ (again, thinking of these as elements of $V^{\QQ_\kappa}$), we see that in~$P$ there is some notion of forcing $\SSS$ such that:
\begin{itemize}
	\item in $P$, $\SSS$ is $<\kappa^\iplus$-distributive;
	\item in $P^{\SSS}$, $j(S)$ reflects no-where at the interval $(\kappa,j(\kappa)]$. 
\end{itemize}
Since $P$ is an inner model of $V^{\QQ_\kappa}$, $\SSS$ is an element of $V^{\QQ_\kappa}$. The model $W$ we are after is $V^{\QQ_\kappa\starr \SSS}$. Because in $V^{\QQ_\kappa}$, $P$ is closed under taking sequences of length $\kappa$, we see that $\SSS$ is $<\kappa^+$-distributive in $V^{\QQ_\kappa}$ as well. This means that $(V_{\kappa+1})^{\QQ_\kappa} = (V_{\kappa+1})^W$. From this we conclude:
\begin{itemize}
	\item In $W$, $S$ is a stationary subset of $\kappa$ and $\funnyDiamond(S)$ holds.
	\item $\mu$ is a normal ultrafilter on $\kappa$  in $W$ as well, and by taking the ultrapower $W^\kappa/\mu$, the embedding $j$ can be extended in $W$ to an embedding (which we also call $j$) from $W$ to an inner model $N$ of $W$. 
\end{itemize}
Finally, since $P\subset V^{\QQ_\kappa}$, $P^{\SSS}\subset W$. As non-reflection is upward absolute, $j(S)$ reflects no-where at the interval $(\kappa,j(\kappa)]$ in $W$ as well. This concludes the proof.

%
%

\end{document}